\documentclass[11pt,twoside]{amsart}
\usepackage{amsmath,amssymb,amsthm}

\newtheorem{thm}{Theorem}[section]

 \newtheorem{cor}{Corollary}[section]
 \newtheorem{lem}{Lemma}[section]
 \newtheorem{prop}{Proposition}[section]
 \newtheorem{defn}{Definition}[section]
\newtheorem{rem}{Remark}[section]

\def\Id{{\rm Id}\,}
\def\d{\partial}

\def\ddk{\dot \Delta_k}

\def\tilde{\widetilde}

\newcommand\R{\mathbb{R}}

\newcommand\Z{\mathbb{Z}}

\newcommand{\N}{\mathbb{N}}

\renewcommand{\div}{\mbox{\rm div}\;\!}

\def\cA{{\mathcal A}}

\def\cD{{\mathcal D}}

\def\cP{{\mathcal P}}
\def\cQ{{\mathcal Q}}

\begin{document}
\title[Compressible MHD system]{Large time behavior of strong solutions to the compressible magnetohydrodynamic system in the critical $L^{p}$ framework}

\author{Wei Xuan Shi}
\address{Department of Mathematics,  Nanjing
University of Aeronautics and Astronautics,
Nanjing 211106, P.R.China,}
\email{wxshi168@163.com}

\author{Jiang Xu}
\address{Department of Mathematics,  Nanjing
University of Aeronautics and Astronautics,
Nanjing 211106, P.R.China,}
\email{jiangxu\underline{ }79@nuaa.edu.cn}
\thanks{The second author is partially supported by the National
Natural Science Foundation of China (11471158), the Program for New Century Excellent
Talents in University (NCET-13-0857) and the Fundamental Research Funds for the Central
Universities (NE2015005).  He would like to thank Professor R. Danchin for his kind communication when visiting the LAMA in UPEC}

\subjclass{76W15, 35Q30, 35L65, 35K65}
\keywords{Compressible MHD system; Decay estimates; Critical Besov spaces}

\begin{abstract}
In this paper, we are concerned with the compressible viscous magnetohydrodynamic system (MHD) and investigate the large time behavior of strong solutions
near constant equilibrium (away from vacuum). In the eighties, Umeda, Kawashima \& Shizuta \cite{UKS} initiated the dissipative mechanism for
a rather general class of symmetric hyperbolic-parabolic systems, which is
$$\mathrm{Re}\lambda\left(i\xi\right)\leq-C\eta\left(\xi\right),\ \ \
\eta\left(\xi\right)=\frac{\left|\xi\right|^{2}}{1+\left|\xi\right|^2}\left(\left|\xi\right|\neq0,\ C>0\right).$$
Here $\lambda=\lambda\left(i\xi\right)$ is the characteristic root of linearized equations. From the point of view of dissipativity, Kawashima \cite{KS} in his doctoral dissertation established the optimal time-decay estimates of $L^{q}$-$L^{2}(q\in [1,2)$) type for solutions to the MHD system. Here, by using Fourier analysis techniques, we shall improve Kawashima's efforts in \cite{KS} and give more precise description for the large-time asymptotic behavior of solutions, not only in extra Lebesgue spaces but also in a full family of Besov norms with the negative regularity index. Precisely, we show that the $L^{p}$ norm (the slightly stronger $\dot B^{0}_{p,1}$ norm in fact) of global solutions with the critical regularity, decays like $t^{-d\left(\frac{1}{p}-\frac{1}{4}\right)}$ for $t\rightarrow\infty$. In particular, taking $p=2$ and $d=3$ goes back to the classical time decay $t^{-\frac{3}{4}}$. We derive new estimates which are used to deal with the strong coupling between the magnetic field and fluid dynamics.
\end{abstract}

\maketitle
\section{Introduction}
The present paper is devoted to the equations of magnetohydrodynamics  which describe the motion of electrically conducting
fluids in the presence of a magnetic field. It involves various topics such as the evolution and dynamics of astrophysical objects, thermonuclear fusion, metallurgy and semiconductor crystal growth, see for example \cite{CH,LL}. Precisely, the barotropic compressible viscous MHD system reads as
\begin{equation} \label{Eq:1-1}
\left\{
\begin{array}{l}
\partial _{t}\varrho +\mathrm{div}\,\left( \varrho u\right) =0, \\
\partial _{t}( \varrho u) +\mathrm{div}\,\left( \varrho u\otimes u\right) +\nabla P(\varrho)\\
\hspace{20mm}=B\cdot\nabla B-\frac{1}{2}\nabla(|B|^2)
+\mathrm{div}\left( 2\mu\, D(u) +\lambda\,\mathrm{div}\,u\, \mathrm{Id}\right), \\
\partial _{t}B+(\div u)B+u\cdot\nabla B-B\cdot\nabla u=\nu\Delta B, \ \ \mathrm{div}\,B=0
\end{array}
\right.
\end{equation}
for $(t,x)\in \mathbb{R}_{+}\times \mathbb{R}^{d}(d\geq2)$. Here, $\varrho =\varrho (t,x)\in \R_{+}$ denotes the density,
$u=\left(u_{1},u_{2},\cdots, u_{d}\right)\in \mathbb{R}^{d}$ is
 the velocity field, and $B=\left(B_{1},B_{2},\cdots,B_{d}\right)\in \mathbb{R}^{d}$ stands for the
magnetic field. The barotropic assumption means that the pressure $P$ is given suitably smooth function of $\varrho$.
 The notation  $D(u)\triangleq\frac{1}{2}(D_{x} u+\left(D_{x} u\right)^{T})$ stands for  the {\it deformation tensor}, and $\div$
 is the divergence operator with respect to the space variable. The  density-dependent functions $\lambda$ and $\mu$  (the  \emph{bulk} and \emph{shear viscosities}) are supposed to be smooth enough and to satisfy $\mu >0$ and $\lambda +2\mu >0$. The
constant $\nu >0$ is the resistivity acting as the magnetic diffusion coefficient of the magnetic field.

System \eqref{Eq:1-1} is supplemented with the initial data
\begin{equation}\label{Eq:1-2}
\left(\varrho,u,B\right)|_{t=0}=\left(\varrho _{0}(x),u_{0}(x),B_{0}(x)\right),\  x\in \R^{d},
\end{equation}
and we focus on solutions that are close to some constant state $\left(\varrho _{\infty}
,0,B_{\infty}\right)$ with $\varrho _{\infty} >0$ and the nonzero vector $B_{\infty}=\left(B^{1}_{\infty},B^{2}_{\infty},\cdots,B^{d}_{\infty}\right)\in \mathbb{R}^{d}$, at infinity.

Due to the physical importance and mathematical challenges, there are lots of works on the existence and uniqueness of (weak, strong or smooth) solutions, see \cite{GMWZ, HT, HW1, HW2, HW3, JJL, KS, LYP, SH} and references therein. Umeda, Kawashima \& Shizuta \cite{UKS} first considered a rather general class of symmetric hyperbolic-parabolic systems in the 80ies. By exploiting an energy method in Fourier spaces, they found the system is
uniformly dissipative in the following sense
$$\mathrm{Re}\lambda\left(i\xi\right)\leq-C\eta(\xi),\ \ \ \eta\left(\xi\right)=\frac{\left|\xi\right|^{2}}{1+\left|\xi\right|^2}\ \ \left(\left|\xi\right|\neq0, \ C>0\right),$$
where $\lambda=\lambda\left(i\xi\right)$ is the characteristic root of linearized equations. Consequently, the dissipative mechanism indicates the optimal decay rates are just the same as that of heat kernel. As direct applications, they showed such decay property of solutions to the linearized form of \eqref{Eq:1-1}-\eqref{Eq:1-2} (near the equilibrium state $\left(\varrho _{\infty},0,B_{\infty}\right)$). Subsequently, Kawashima \cite{KS} proved the global existence of smooth solutions to \eqref{Eq:1-1}-\eqref{Eq:1-2} by assuming the initial data are small in $H^3\left(\mathbb{R}^{3}\right)$. Moreover, he obtained the following fundamental $L^q$-$L^2$ decay estimate in $H^{4}\left(\mathbb{R}^{3}\right)\cap L^{q}\left(\mathbb{R}^{3}\right)$ $\left(q\in[1,2)\right)$:
\begin{equation} \label{Eq:1-3}
\left\|\left(\varrho-\varrho_{\infty},u,B-B_{\infty}\right)(t)\right\|_{L^{2}(\mathbb{R}^{3})}\leq C \langle t\rangle^{-\frac32\left(\frac{1}{q}-\frac{1}{2}\right)}
\ \ \hbox{with} \ \ \langle t\rangle\triangleq\sqrt{1+t^{2}}.
\end{equation}
Shortly after Kawashima, still for data with high Sobolev regularity, there are a number of works on that topic, see for example
\cite{CT,GCY,LY,TW} and references therein. Hu and Wang \cite{HW1} established the existence and large-time behavior of global weak solutions with large data in
a bounded domain. The global existence of large-weak solutions was also shown for full compressible magnetohydrodynamic flows where the
viscosity coefficients and heat conductivity could depend on the temperature, see \cite{HW2}. Suen and Hoff \cite{SH} obtained the global low-energy weak solutions where initial data are chosen to be small and initial densities are assumed to be nonnegative and essentially bounded. Hoff and Tsyganov \cite{HT} studied the uniqueness and continuous dependence of weak solutions in compressible magnetohydrodynamics. Hu and Wang \cite{HW3} investigated the low-Mach-limit problem for isentropic MHD system in the framework of weak solutions, and Jiang et al. \cite{JJL} obtained the low-Mach-limit for smooth solutions. The definite convergence order was shown by Li in \cite{LYP}. The existence and stability of non-characteristic boundary layers for the compressible MHD equations were proved by Gu\`{e}s et al. \cite{GMWZ}.

On the other hand, notice that $B\equiv0$, system \eqref{Eq:1-1} reduces to the usual compressible Navier-Stokes system for baratropic fluids. Inspired by
those efforts on the incompressible Navier-Stokes equations \cite{CM,FK, KY}, Danchin \cite{DR1} performed the scaling invariance and established
the global existence of strong solutions when the initial data are small enough in $L^{2}$ critical homogeneous Besov space. Later, Charve \& Danchin \cite{CD},
Chen, Miao \& Zhang \cite{CMZ} independently extended that result to
more general $L^{p}$ critical Besov spaces. In \cite{HB}, Haspot employed $L^{p}$ energy methods which is based on the use of Hoff's viscous effective flux in \cite{HD}, and got the essentially same results as in \cite{CD,CMZ}. Recently, Danchin and the second author \cite{DX} proved the large-time behavior of strong solutions in the $L^{p}$ critical regularity framework. To the best of our knowledge, there are few results (except for \cite{HCH}) for the magnetohydrodynamic system from the point of view of scaling invariance. In fact, observe that system \eqref{Eq:1-1} is obviously invariant for all $l>0$ by the
following transformation
\begin{equation*}
\varrho ( t,x) \rightsquigarrow \varrho \left( l^{2}t,l x\right) ,\ \ u(t,x) \rightsquigarrow l u\left( l^{2}t,l x\right) ,\ \ B(t,x) \rightsquigarrow l B\left( l ^{2}t,l x\right),
\end{equation*}
up to a change of the pressure law $P$ into $l^{2}P$. Inspired by the scaling property, Hao \cite{HCH} achieved the global existence of strong solutions to \eqref{Eq:1-1}-\eqref{Eq:1-2} in $L^{2}$ critical Besov spaces by using the compactness arguments. Recently,
the authors in another paper studied the coupling effect arising from the magnetic field and established the global strong solutions in
more general $L^{p}$ framework. Precisely, one has
\begin{thm}\label{Thm1.1}(\cite{SX})
 Let $d\geq2$ and $p$ fulfill
\begin{equation}\label{Eq:1-4}
2\leq p\leq \min\left(4,2d/(d-2)\right) \ and, \ additionally, \ p\neq 4 \ if \ d=2.
\end{equation}
Suppose that $\div B_0=0$, $P^{\prime }\left(\varrho_{\infty} \right)>0$ and that \eqref{Eq:1-2} is satisfied. There exists a small positive constant $c=c\left( p,\mu ,\lambda,\nu ,P,\varrho_{\infty},B_{\infty}\right) $ and a universal integer $k_{0}\in
\mathbb{N}$ such that if $a_{0}\triangleq \varrho_{0}-\varrho_{\infty}$ is in $\dot{B}_{p,1}^{\frac {d}{p}}$, if $u_{0}$ is in $\dot{B}_{p,1}^{\frac {d}{p}-1}$, if $H_{0}\triangleq B_{0}-B_{\infty}$ is in $\dot{B}_{p,1}^{\frac {d}{p}-1}$ and if in addition $\left(a_{0}^{\ell} ,u_{0}^{\ell}
,H_{0}^{\ell}\right)$ in $\dot{B}_{2,1}^{\frac {d}{2}-1}$ \emph{(}with the notation $z^{\ell} =%
\dot{S}_{k_{0}+1}z$ and $z^{h}=z-z^{\ell} $\emph{)} with
\begin{equation*}
\mathcal{E}_{p,0}\triangleq \left\|\left( a_{0},u_{0},H_{0}\right) \right\|_{\dot{B}_{2,1}^{\frac {d}{2}-1}}^{\ell} +\left\|a_{0}\right\|_{\dot{B}_{p,1}^{\frac
{d}{p}}}^{h}+\left\|\left(u_{0},H_{0}\right)\right\|_{\dot{B}_{p,1}^{\frac {d}{p}-1}}^{h}\leq c,
\end{equation*}
then the Cauchy problem \eqref{Eq:1-1}-\eqref{Eq:1-2} admit a unique global-in-time solution $\left(\varrho,u,B\right)$ with $\varrho=\varrho_{\infty}+a$, $B=B_{\infty}+H$ and $\left(a,u,H\right)$
in the space $X_{p}$ defined by:
\begin{eqnarray*}
&&\hspace{-6mm}\left(a,u,H\right)^{\ell} \in \widetilde{\mathcal{C}_{b}}\left(\mathbb{R_{+}};\dot{B}_{2,1}^{\frac {d}{2}-1}\right)\cap L^{1}\left(\mathbb{R_{+}};\dot{B}_{2,1}^{\frac {d}{2}+1}\right),\, a^{h}\in \widetilde{\mathcal{C}_{b}}\left(\mathbb{R_{+}};\dot{B}_{p,1}^{\frac{d}{p}}\right)\cap L^{1}\left(\mathbb{R_{+}};\dot{B}_{p,1}^{\frac {d}{p}}\right),\\
&&\hspace{-6mm}\left(u,H\right)^{h}\in \widetilde{\mathcal{C}_{b}}\left(\mathbb{R_{+}};\dot{B}_{p,1}^{\frac {d}{p}-1}\right)\cap L^{1}\left(\mathbb{R_{+}};\dot{B}_{p,1}^{\frac {d}{p}+1}\right),
\end{eqnarray*}
where we agree that $\tilde{\mathcal{C}}_{b}\left(\mathbb{R}_{+};\dot{B}^{s}_{q,1}\right)\triangleq\mathcal{C}\left(\mathbb{R}_{+};\dot{B}^{s}_{q,1}\right)
\cap\tilde{L}^{\infty}\left(\mathbb{R}_{+};\dot{B}^{s}_{q,1}\right)$, $s\in\mathbb{R}$, $1\leq q \leq \infty$.

Furthermore, we get for some constant $C=C\left( p,\mu ,\lambda,\nu ,P,\varrho_{\infty},B_{\infty}\right)$,
\begin{equation*}
\mathcal{E}_{p}(t)\leq C\mathcal{E}_{p,0},
\end{equation*}
for any $t>0$, where
\begin{eqnarray}
&&\mathcal{E}_{p}(t)\triangleq\left\| (a,u,H)\right\|_{\widetilde{L}^{\infty} (\dot{B}
_{2,1}^{\frac {d}{2}-1})}^{\ell}+\left\| (a,u,H)\right\| _{L^{1}(\dot{B}
_{2,1}^{\frac {d}{2}+1})}^{\ell}+\left\|a\right\|_{\widetilde{L}^{\infty}(\dot{B}_{p,1}^{\frac {d}{p}})}^h+\left\|a\right\|_{L^{1}(\dot{B}_{p,1}^{\frac
{d}{p}})}^{h} \nonumber \\\label{Eq:1-5}
&&\hspace{13.5mm}+\left\|(u,H)\right\|_{\widetilde{L}^{\infty}(\dot{B}_{p,1}^{\frac {d}{p}-1})}^{h}+\left\|(u,H)\right\| _{L^{1}(\dot{B}_{p,1}^{\frac {d}{p}+1})}^{h}.
\end{eqnarray}
\end{thm}
\begin{rem}\label{Rem1.1}
Condition \eqref{Eq:1-4} allows to the case  $p>d$ for which the
regularity exponent $d/p-1$ of $(u,H)$ becomes negative in physical dimensions $d=2,3$, so our result holds for small data but
large highly oscillating initial velocity and magnetic field. This is an \textit{essential} and \textit{different} point in comparison with $L^2$ results in \cite{HCH}.
On the other hand, the regularities of global solutions are reduced heavily in contrast with all previous works \cite{KS,CT,GCY,LY,TW} and references therein.
\end{rem}

The natural next step is to look for the large-time asymptotic description of the global solution. For that purpose, let's rewrite system \eqref{Eq:1-1} as the nonlinear perturbation form of constant equilibrium state $\left(\varrho_{\infty},0,B_{\infty}\right)$, looking at the nonlinearities as source terms. To
simplify the statement of main results, we assume that $\varrho_{\infty}=1$, $B_{\infty}=I$ ($I$ is an arbitrary nonzero constant vector satisfying $\left|I\right|=1$), $P'(\varrho_{\infty})=1$, $\nu=1$ and $\nu_{\infty}\triangleq2\mu_{\infty}+\lambda_{\infty}=1$ (with $\lambda_{\infty}\triangleq \lambda\left(\varrho_{\infty}\right)$ and $\mu_{\infty}\triangleq \mu\left(\varrho_{\infty}\right)$). Consequently,
in term of the new variables $\left(a,u,H\right)$, system \eqref{Eq:1-1} becomes
\begin{equation} \label{Eq:1-6}
\left\{
\begin{array}{l}
\partial _{t}a+\div u=f, \\
\partial _{t}u-\mathcal{A}u+\nabla a+\nabla\left(I\cdot H\right)-I\cdot\nabla H=g, \\
\partial _{t}H-\Delta H+\left(\div u\right) I-I\cdot\nabla u=m,\\
\div H=0,
\end{array}
\right.
\end{equation}
where
\begin{eqnarray*}
&&f \triangleq -\div \left(au\right), \\
&&g \triangleq -u\cdot \nabla u-\pi_{1}(a)\mathcal{A}u-\pi_{2}(a)\nabla a+\frac {1}{1+a}\mathrm{div}\left(2\widetilde{\mu }(a)\,D(u)+\widetilde{\lambda }(a) \div u \,\Id \right)\\
&&\hspace{7.5mm}+\pi_{1}(a)\left(\nabla\left(I\cdot H\right)-I\cdot\nabla H\right)-\frac {1}{1+a}\left(\frac{1}{2}\nabla|H|^{2}-H\cdot\nabla H\right),\\
&&m\triangleq -H\left(\div u\right)+H\cdot \nabla u-u\cdot \nabla H,
\end{eqnarray*}
with
\begin{center}
$\mathcal{A}\triangleq\mu _{\infty} \Delta +\left(\lambda _{\infty} +\mu _{\infty}\right)
\nabla \div$ such that $2\mu_{\infty}+\lambda _{\infty}=1$ and
$\mu _{\infty} >0$ \\
($\mu _{\infty}\triangleq\mu (1)$ and $\lambda _{\infty} \triangleq\lambda (1)$), \ \ $\pi_{1}(a) \triangleq \frac {a}{1+a}$, \ \
$\pi_{2}( a) \triangleq \frac{P^{\prime}( 1+a)}{1+a}-1$, \\
$\widetilde{\mu }(a)\triangleq \mu (1+a)-\mu (1)$, \ \ $\widetilde{\lambda}(a)\triangleq \lambda(1+a)-\lambda (1)$.
\end{center}
Note that $\pi_{1}, \pi_{2}, \ \widetilde{\mu }$ and $\widetilde{\lambda}$ are smooth functions satisfying
$$\pi_{1}(0)=\pi_{2}(0)=\widetilde{\mu }(0)=\widetilde{\lambda}(0)=0.$$

Motivated by the dissipative analysis due to
Kawashima et al. in \cite{UKS}, we aim at establishing the large-time behavior of strong solutions constructed in Theorem \ref{Thm1.1}. The main result is stated as follows.
\begin{thm}\label{Thm1.2}
Let $d\geq2$ and $p$ satisfy assumption \eqref{Eq:1-4}. Let $\left(\varrho_{0},u_{0},B_{0}\right)$ fulfill the conditions of Theorem \emph{\ref{Thm1.1}}, and denote by $\left(\varrho,u,B\right)$ the corresponding global solution of system \eqref{Eq:1-1}. There exists a positive constant $c=c\left( p,\mu ,\lambda,P,B_{\infty}\right)$ such that if
\begin{equation} \label{Eq:1-7}
\mathcal{D}_{p,0}\triangleq \left\|\left(a_{0},u_{0},H_{0}\right)\right\|^{\ell}_{\dot{B}_{2,\infty}^{-s_{0}}}\leq c \ \ with  \ \ s_{0}\triangleq \frac{2d}{p}-\frac{d}{2},
\end{equation}
then we have for all $t\geq 0$,
\begin{equation}\label{Eq:1-8}
\mathcal{D}_{p}(t)\lesssim \left(\mathcal{D}_{p,0}+\left\|\left(\nabla a_{0},u_{0},H_{0}\right)\right\|^{h}_{\dot{B}_{p,1}^{\frac {d}{p}-1}}\right),
\end{equation}
where the functional $\mathcal{D}_{p}(t)$ is defined by
\begin{eqnarray}
&&\mathcal{D}_{p}(t)\triangleq \sup_{s\in[\varepsilon-s_{0},\frac{d}{2}+1]}
\left\|\langle\tau\rangle^{\frac{s_{0}+s}{2}}\left(a,u,H\right)\right\|^{\ell}_{L^{\infty}_{t}(\dot{B}_{2,1}^{s})}
+\left\|\langle\tau\rangle^{\alpha}\left(\nabla a,u,H\right)\right\|^{h}_{\tilde{L}^{\infty}_{t}(\dot{B}_{p,1}^{\frac {d}{p}-1})} \nonumber\\ \label{Eq:1-9}
&&\hspace{14mm}+\left\|\tau \nabla \left(u,H\right)\right\|^{h}_{\tilde{L}^{\infty}_{t}(\dot{B}_{p,1}^{\frac {d}{p}})}
\end{eqnarray}
with $\alpha \triangleq \frac{d}{p}+\frac{1}{2}-\varepsilon$ for some sufficiently small $\varepsilon>0$.
\end{thm}
\begin{rem}\label{Rem1.2}
If replacing \eqref{Eq:1-7} by the following slightly stronger condition:
$$\left\|\left(a_{0},u_{0},H_{0}\right)\right\|^{\ell}_{\dot{B}^{-s_{0}}_{2,1}}\leq c\ll1,$$
then we can take $\varepsilon=0$ in $\alpha$,  and the first term of $\mathcal{D}_{p}(t)$ is correspondingly replaced by
$$\sup_{s\in\left[-s_{0},\frac{d}{2}+1\right]}\left\|\langle\tau\rangle^{\frac{s_{0}+s}{2}}
\left(a,u,H\right)\right\|^{\ell}_{\tilde{L}^{\infty}_{t}(\dot{B}^{s}_{2,1})}.$$
\end{rem}
In the case of high Sobolev regularity, see for example \cite{KS}, one usual get the optimal time-decay estimate $L^{1}$-$L^{2}$ as in \eqref{Eq:1-3} by using
the combination of linearized spectral analysis and Duhamel principle. Here let's state the major difficulty and the strategy to overcome it. More concrete speaking, by the rough spectral analysis in our recent paper \cite{SX}, we know system
\eqref{Eq:1-6} has a parabolic smoothing effect on $u$, $H$ as well as the low frequencies of $a$, and a damping effect on the high
frequencies of $a$. Obviously, there is a loss of one derivative of density (see the convection term $u\cdot\nabla a$ in the transport equation), but
one cannot afford any loss of regularity for the high-frequency part of density. Consequently, Duhamel principle cannot remain valid in the critical regularity framework. To overcome the difficulty, we shall use the effective velocity (introduced by Haspot in \cite{HB}, which is another name for the
\textit{viscous effective flux} of Hoff \cite{HD}) to get the estimates of high frequencies, which allows to afford the loss of one derivative coming from the density in the transport equation. On the other hand, due to the strong interaction between the magnetic fluid and the hydrodynamic motion, the spectral analysis becomes more complicated mathematically. In order to get $L^{2}$ decay estimates for the low frequency at lower price, alternatively, one may attempt to use the general theory of partially dissipative first-order symmetric systems initiated by Godunov \cite{GSK}, then well developed by Kawashima et al. \cite{SK}. Although the current proofs are in spirit of the joint work with Danchin for compressible Navier-Stokes equations (see \cite{DX}), the strong coupling between the magnetic field and fluid dynamics need to be handled with. In particular, due to non-zero equilibrium of magnetic field in comparison with \cite{CT,GCY,LY,TW}, we take care of extra
linear coupling terms (see \eqref{Eq:1-6}) in energy approaches.

In addition, we give some consequences of Theorem \ref{Thm1.2}.
\begin{cor}\label{Cor1.1}
The solution $(\varrho,u,B)$ constructed in Theorem \emph{\ref{Thm1.2}} fulfills
$$\left\|\Lambda^{s}\left(\varrho-1\right)\right\|_{L^{p}}\lesssim  \left(\mathcal{D}_{p,0}+\left\|\left(\nabla a_{0},u_{0},H_{0}\right)\right\|^{h}_{\dot{B}^{\frac{d}{p}-1}_{p,1}}\right)\langle t\rangle^{-\frac{s_{0}+s}{2}} \text{ if }-s_{0}<s\leq\frac{d}{p},$$
$$\left\|\Lambda^{s}\left(u,B-I\right)\right\|_{L^{p}}\lesssim \left(\mathcal{D}_{p,0}+\left\|\left(\nabla a_{0},u_{0},H_{0}\right)\right\|^{h}_{\dot{B}^{\frac{d}{p}-1}_{p,1}}\right)\langle t\rangle^{-\frac{s_{0}+s}{2}} \text{ if }-s_{0}<s\leq\frac{d}{p}-1,$$
where the pseudo differential operator $\Lambda^{\ell}$ is defined by $\Lambda^{\ell}f\triangleq\mathcal{F}^{-1}\left(|\xi|^{\ell}\mathcal{F}f\right)$.
\end{cor}
\begin{rem}\label{Rem1.3}
Taking $p=2$ \emph{(}$s_{0}=\frac{d}{2}$\emph{)} and $s=0$ in Corollary \emph{\ref{Cor1.1}} leads back to the standard optimal $L^{1}$-$L^{2}$ decay rate of $\left(\varrho-1,u,B-I\right)$, which was firstly shown by \cite{KS} and developed by \emph{\cite{CT,GCY,LY,TW}} etc..
The harmonic analysis allows us to establish the optimal decay estimates in lower regularity framework.
Moreover, it is observed that the derivative exponent $s$ can take values in some interval, and hence there is no derivative loss for those decay estimates.
\end{rem}
One can get more $L^{q}$-$L^{r}$ decay estimates with aid of Gagliardo-Nirenberg interpolation inequalities which parallel the work of Sohinger and Strain \cite{SS}
(see also \cite{BCD}, Chap. 2 or \cite{XK1,XK2}).
\begin{cor}\label{Cor1.2}
 Let the assumptions of Theorem \emph{\ref{Thm1.2}} be satisfied with $p=2$. Then the corresponding solution $\left(\varrho,u,B\right)$ fulfills
$$\left\|\Lambda^{l}\left(\varrho-1,u,B-I\right)\right\|_{L^{r}}\lesssim \left(\mathcal{D}_{2,0}+\left\|\left(\nabla a_{0},u_{0},H_{0}\right)\right\|^{h}_{\dot{B}^{\frac{d}{2}-1}_{2,1}}\right)
\langle t\rangle^{-\frac{d}{2}\left(1-\frac{1}{r}\right)-\frac{l}{2}},$$
for all $2\leq r\leq\infty$ and $l\in\mathbb{R}$ satisfying $-\frac{d}{2}<l+d\left(\frac{1}{2}-\frac{1}{r}\right)\leq\frac{d}{2}-1$.
\end{cor}
The rest of the paper unfolds as follows: In section 2, we briefly recall Littlewood-Paley decomposition, Besov spaces  and related analysis tools.
Section 3 is devoted to the proofs of Theorem \ref{Thm1.2} and Corollaries \ref{Cor1.1}-\ref{Cor1.2}.
\section{Preliminary}\setcounter{equation}{0}
Throughout the paper, $C>0$ stands for a harmless ``constant''. For brevity, we sometime write
$f\lesssim g$ instead of $f\leq Cg$. The notation $f\approx g$ means that $%
f\lesssim g$ and $g\lesssim f$. For any Banach space $X$ and $f,g\in X$, we agree that
$\left\|\left(f,g\right)\right\| _{X}\triangleq \left\|f\right\| _{X}+\left\|g\right\|_{X}$. For
all $T>0$ and $\rho\in\left[1,+\infty\right]$, we denote by
$L_{T}^{\rho}(X) \triangleq L^{\rho}\left(\left[0,T\right];X\right)$ the set of measurable functions $f:\left[0,T\right]\rightarrow X$ such that $t\mapsto\left\|f(t)\right\|_{X}$ is in $L^{\rho}\left(0,T\right)$.

To make the paper self-contained, let us briefly recall Littlewood-Paley decomposition, Besov spaces and related analysis tools. More details
may be found for example in Chap. 2 and Chap. 3 of \cite{BCD}. Firstly, we introduce a homogeneous Littlewood-Paley decomposition. To do this, we fix some smooth
radial non increasing function $\chi $ with $\mathrm{Supp}\,\chi \subset
B\left(0,\frac {4}{3}\right)$ and $\chi \equiv 1$ on $B\left(0,\frac
{3}{4}\right)$, then set $\varphi \left(\xi\right) =\chi \left(\xi/2\right)-\chi \left(\xi\right)$
so that
\[
\sum_{k\in \mathbb{Z}}\varphi \left( 2^{-k}\cdot \right) =1\ \ \text{in}\ \
\mathbb{R}^{d}\setminus \{ 0\} \ \ \text{and}\ \ \mathrm{Supp}\,\varphi \subset \left\{ \xi \in \mathbb{R}^{d}:3/4\leq |\xi|\leq 8/3\right\} .
\]
Then homogeneous dyadic blocks $\dot{\Delta}_k$ are defined by
\[
\dot{\Delta}_{k}f\triangleq \varphi (2^{-k}D)f=\mathcal{F}^{-1}(\varphi
(2^{-k}\cdot )\mathcal{F}f)=2^{kd}h(2^{k}\cdot )\star f\ \ \text{with}\ \
h\triangleq \mathcal{F}^{-1}\varphi .
\]
Therefore, one has the following homogeneous decomposition of a tempered distribution $f\in S^{\prime }(\mathbb{R}^{d})$
\begin{equation} \label{Eq:2-1}
f=\sum_{k\in \mathbb{Z}}\dot{\Delta}_{k}f.
\end{equation}
As it holds only modulo polynomials, it is convenient to consider the subspace of those tempered distributions $f$ such that
\begin{equation}\label{Eq:2-2}
\lim_{k\rightarrow -\infty }\| \dot{S}_{k}f\| _{L^{\infty} }=0,
\end{equation}
where $\dot{S}_{k}f$ stands for the low frequency cut-off defined by $\dot{S}_{k}f\triangleq\chi \left(2^{-k}D\right)f$. Indeed,
if \eqref{Eq:2-2} is fulfilled, then \eqref{Eq:2-1} holds true in $S'\left(\mathbb{R}^{d}\right)$. For convenience, we denote by $S'_{0}\left(\mathbb{R}^{d}\right)$ the subspace of tempered distributions satisfying \eqref{Eq:2-2}.

With aid of the Littlewood-Paley decomposition, Besov spaces
and related analysis tools will come into play in our context.
\begin{defn}\label{Defn2.1}
For $s\in \mathbb{R}$ and $1\leq p,r\leq\infty,$ the homogeneous
Besov spaces $\dot{B}^{s}_{p,r}$ is defined by
$$\dot{B}^{s}_{p,r}\triangleq\left\{f\in S'_{0}:\left\|f\right\|_{\dot{B}^{s}_{p,r}}<+\infty\right\},$$
where
\begin{equation}\label{Eq:2-3}
\left\|f\right\|_{\dot B^{s}_{p,r}}\triangleq\left\|\left(2^{ks}\|\ddk  f\|_{L^p}\right)\right\|_{\ell^{r}(\Z)}.
\end{equation}
\end{defn}
When studying the evolution PDEs, a class of mixed space-time Besov spaces are also used, which was first introduced by J.-Y. Chemin and N. Lerner \cite{CL} (see also \cite{CJY} for the particular case of Sobolev spaces).
\begin{defn}\label{Defn2.2}
 For $T>0, s\in\mathbb{R}, 1\leq r,\theta\leq\infty$, the homogeneous Chemin-Lerner space $\widetilde{L}^{\theta}_{T}(\dot{B}^{s}_{p,r})$
is defined by
$$\widetilde{L}^{\theta}_{T}(\dot{B}^{s}_{p,r})\triangleq\left\{f\in L^{\theta}\left(0,T;S'_{0}\right):\left\|f\right\|_{\widetilde{L}^{\theta}_{T}(\dot{B}^{s}_{p,r})}<+\infty\right\},$$
where
\begin{equation}\label{Eq:2-4}
\left\|f\right\|_{\widetilde{L}^{\theta}_{T}(\dot{B}^{s}_{p,r})}\triangleq\left\|\left(2^{ks}\|\ddk  f\|_{L^{\theta}_{T}(L^{p})}\right)\right\|_{\ell^{r}(\Z)}.
\end{equation}
\end{defn}
For notational simplicity, index $T$ is omitted if $T=+\infty $.
We also use the following functional space:
\begin{equation*}
\tilde{\mathcal{C}}_{b}\left(\mathbb{R_{+}};\dot{B}_{p,r}^{s}\right)\triangleq \left\{f \in
\mathcal{C}\left(\mathbb{R_{+}};\dot{B}_{p,r}^{s}\right)\ \hbox{s.t}\ \left\|f\right\| _{\tilde{L}^{\infty}(\dot{B}_{p,r}^{s})}<+\infty \right\} .
\end{equation*}
The above norm \eqref{Eq:2-4} may be compared with those of the standard spaces $L_{T}^{\rho} (\dot{B}_{p,r}^{s})$ by means of Minkowski's
inequality.
\begin{rem}\label{Rem2.1}
It holds that
$$\left\|f\right\|_{\widetilde{L}^{\theta}_{T}(B^{s}_{p,r})}\leq\left\|f\right\|_{L^{\theta}_{T}(B^{s}_{p,r})}\,\,\,
\mbox{if} \,\, \, r\geq\theta;\ \ \ \
\left\|f\right\|_{\widetilde{L}^{\theta}_{T}(B^{s}_{p,r})}\geq\left\|f\right\|_{L^{\theta}_{T}(B^{s}_{p,r})}\,\,\,
\mbox{if}\,\,\, r\leq\theta.
$$\end{rem}
Restricting the above norms \eqref{Eq:2-3} and \eqref{Eq:2-4} to the low or high
frequencies parts of distributions will be fundamental in our approach. For example, let us fix some integer $k_{0}$ (the value of which will
follow from the proof of the main theorem) and put\footnote{Note that for technical reasons, we need a small
overlap between low and high frequencies.}
$$\left\| f\right\| _{\dot{B}_{p,1}^{s}}^{\ell} \triangleq \sum_{k\leq
k_{0}}2^{ks}\left\| \dot{\Delta}_{k}f\right\|_{L^{p}} \ \mbox{and} \ \left\|f\right\|_{\dot{B}_{p,1}^{s}}^{h}\triangleq \sum_{k\geq k_{0}-1}2^{ks}\left\| \dot{\Delta}_{k}f\right\| _{L^{p}},$$
$$\left\|f\right\| _{\tilde{L}_{T}^{\infty} (\dot{B}_{p,1}^{s})}^{\ell} \triangleq
\sum_{k\leq k_{0}}2^{ks}\left\|\dot{\Delta}_{k}f\right\|_{L_{T}^{\infty} (L^{p})} \
\mbox{and} \ \left\|f\right\| _{\tilde{L}_{T}^{\infty} (\dot{B}_{p,1}^{s})}^{h}\triangleq \sum_{k\geq k_{0}-1}2^{ks}\left\| \dot{\Delta}_{k}f\right\|
_{L_{T}^{\infty} (L^{p})}.$$

We often use the following classical properties (see \cite{BCD}):

$\bullet$ \ \emph{Scaling invariance:} For any $s\in \mathbb{R}$ and $(p,r)\in
[1,\infty ]^{2}$, there exists a constant $C=C(s,p,r,d)$ such that for all $\lambda >0$ and $f\in \dot{B}_{p,r}^{s}$, we have
$$
C^{-1}\lambda ^{s-\frac {d}{p}}\left\|f\right\|_{\dot{B}_{p,r}^{s}}
\leq \left\|f(\lambda \cdot)\right\|_{\dot{B}_{p,r}^{s}}\leq C\lambda ^{s-\frac {d}{p}}\left\|f\right\|_{\dot{B}_{p,r}^{s}}.
$$

$\bullet$ \ \emph{Completeness:} $\dot{B}^{s}_{p,r}$ is a Banach space whenever $%
s<\frac{d}{p}$ or $s\leq \frac{d}{p}$ and $r=1$.

$\bullet$ \ \emph{Interpolation:} The following inequalities are satisfied for $1\leq p,r_{1},r_{2}, r\leq \infty, s_{1}\neq s_{2}$ and $\theta \in (0,1)$:
$$\left\|f\right\|_{\dot{B}_{p,r}^{\theta s_{1}+(1-\theta )s_{2}}}\lesssim \left\|f\right\| _{\dot{B}_{p,r_{1}}^{s_{1}}}^{\theta} \left\|f\right\|_{\dot{B}_{p,r_2}^{s_{2}}}^{1-\theta }.$$

$\bullet$ \ \emph{Action of Fourier multipliers:} If $F$ is a smooth homogeneous of
degree $m$ function on $\mathbb{R}^{d}\backslash \{0\}$ then
$$F(D):\dot{B}_{p,r}^{s}\rightarrow \dot{B}_{p,r}^{s-m}.$$
\begin{prop} \label{Prop2.1} (Embedding for Besov spaces on $\R^{d}$)
\begin{itemize}
\item For any $p\in[1,\infty]$ we have the  continuous embedding
$\dot {B}^{0}_{p,1}\hookrightarrow L^{p}\hookrightarrow \dot {B}^{0}_{p,\infty}.$
\item If $\sigma\in\R$, $1\leq p_{1}\leq p_{2}\leq\infty$ and $1\leq r_{1}\leq r_{2}\leq\infty,$
then $\dot {B}^{\sigma}_{p_1,r_1}\hookrightarrow
\dot {B}^{\sigma-d(\frac{1}{p_{1}}-\frac{1}{p_{2}})}_{p_{2},r_{2}}$.
\item The space  $\dot {B}^{\frac {d}{p}}_{p,1}$ is continuously embedded in the set  of
bounded  continuous functions \emph{(}going to zero at infinity if, additionally, $p<\infty$\emph{)}.
\end{itemize}
\end{prop}
The following product estimates in Besov spaces play a fundamental role in our analysis of the bilinear terms of \eqref{Eq:1-6}.
\begin{prop}\label{Prop2.2}(\cite{BCD, DX})
Let $\sigma>0$ and $1\leq p,r\leq\infty$. Then $\dot{B}^{\sigma}_{p,r}\cap L^{\infty}$ is an algebra and
$$
\left\|fg\right\|_{\dot{B}^{\sigma}_{p,r}}\lesssim \left\|f\right\|_{L^{\infty}}\left\|g\right\|_{\dot{B}^{\sigma}_{p,r}}+\left\|g\right\|_{L^{\infty}}\left\|f\right\|_{\dot{B}^{\sigma}_{p,r}}.
$$
Let the real numbers $\sigma_{1},$ $\sigma_{2},$ $p_1$  and $p_2$ be such that
$$
\sigma_{1}+\sigma_{2}>0,\quad \sigma_{1}\leq\frac {d}{p_{1}},\quad\sigma_{2}\leq\frac {d}{p_{2}},\quad
\sigma_{1}\geq\sigma_{2},\quad\frac{1}{p_{1}}+\frac{1}{p_{2}}\leq1.
$$
Then we have
$$\left\|fg\right\|_{\dot{B}^{\sigma_{2}}_{q,1}}\lesssim \left\|f\right\|_{\dot{B}^{\sigma_{1}}_{p_{1},1}}\left\|g\right\|_{\dot{B}^{\sigma_{2}}_{p_{2},1}}\quad\hbox{with}\quad
\frac1{q}=\frac1{p_{1}}+\frac1{p_{2}}-\frac{\sigma_{1}}d\cdotp$$
Additionally, for exponents $\sigma>0$ and $1\leq p_{1},p_{2},q\leq\infty$ satisfying
$$\frac{d}{p_{1}}+\frac{d}{p_{2}}-d\leq \sigma \leq\min\left(\frac {d}{p_{1}},\frac {d}{p_{2}}\right)\quad\hbox{and}\quad \frac{1}{q}=\frac {1}{p_{1}}+\frac {1}{p_{2}}-\frac{\sigma}{d},$$
we have
$$\left\|fg\right\|_{\dot{B}^{-\sigma}_{q,\infty}}\lesssim\left\|f\right\|_{\dot{B}^{\sigma}_{p_{1},1}}\left\|g\right\|_{\dot{B}^{-\sigma}_{p_{2},\infty}}.$$
\end{prop}
Proposition \ref{Prop2.2} is not enough to bound the possible case $p>d$ in the proof of Theorem \ref{Thm1.2}, so we have the following non-classical product estimates.
\begin{prop}\label{Prop2.3} (\cite{DX}) Let $k_{0}\in\Z,$ and denote $z^{\ell}\triangleq\dot S_{k_{0}}z$, $z^{h}\triangleq z-z^{\ell}$ and, for any $s\in\R$,
$$
\left\|z\right\|_{\dot B^{s}_{2,\infty}}^{\ell}\triangleq\sup_{k\leq k_{0}}2^{ks} \left\|\ddk z\right\|_{L^2}.
$$
There exists a universal integer $N_{0}$ such that  for any $2\leq p\leq 4$ and $\sigma>0,$ we have
\begin{eqnarray}\label{Eq:2-5}
&&\left\|f g^{h}\right\|_{\dot {B}^{-s_{0}}_{2,\infty}}^{\ell}\leq C \left(\left\|f\right\|_{\dot {B}^{\sigma}_{p,1}}+\left\|\dot S_{k_{0}+N_{0}}f\right\|_{L^{{p}^{*}}}\right)\left\|g^{h}\right\|_{\dot{B}^{-\sigma}_{p,\infty}}\\\label{Eq:2-6}
&&\left\|f^{h} g\right\|_{\dot {B}^{-s_{0}}_{2,\infty}}^{\ell}
\leq C \left(\left\|f^{h}\right\|_{\dot{B}^{\sigma}_{p,1}}+\left\|\dot{S}_{k_{0}+N_{0}}f^{h}\right\|_{L^{p^{*}}}\right)\left\|g\right\|_{\dot {B}^{-\sigma}_{p,\infty}}
\end{eqnarray}
with  $s_{0}\triangleq \frac{2d}{p}-\frac {d}{2}$ and $\frac1{p^{*}}\triangleq\frac{1}{2}-\frac{1}{p},$
and $C$ depending only on $k_{0}$, $d$ and $\sigma$.
\end{prop}
System \eqref{Eq:1-6} also involves compositions of functions (through $\pi_{1}(a)$, $\pi_{2}(a)$, $\tilde{\lambda}(a)$ and $\tilde{\mu}(a)$) and they
are bounded according to the following conclusion.
\begin{prop}\label{Prop2.4}
Let $F:\R\rightarrow\R$ be  smooth with $F(0)=0$.
For  all  $1\leq p,r\leq\infty$ and $\sigma>0$ we have
$F(f)\in \dot {B}^{\sigma}_{p,r}\cap L^{\infty}$  for  $f\in \dot {B}^{\sigma}_{p,r}\cap L^{\infty},$  and
$$\left\|F(f)\right\|_{\dot B^\sigma_{p,r}}\leq C\left\|f\right\|_{\dot B^\sigma_{p,r}}$$
with $C$ depending only on $\left\|f\right\|_{L^{\infty}}$, $F'$ (and higher derivatives), $\sigma$, $p$ and $d$.
\medbreak
In the case $\sigma>-\min\left(\frac {d}{p},\frac {d}{p'}\right)$ then $f\in\dot {B}^{\sigma}_{p,r}\cap\dot {B}^{\frac {d}{p}}_{p,1}$
implies that $F(f)\in \dot {B}^{\sigma}_{p,r}\cap\dot {B}^{\frac {d}{p}}_{p,1}$, and we have
$$\left\|F(f)\right\|_{\dot B^{\sigma}_{p,r}}\leq C(1+\left\|f\right\|_{\dot {B}^{\frac {d}{p}}_{p,1}})\left\|f\right\|_{\dot {B}^{\sigma}_{p,r}}.$$
\end{prop}
Let us now  recall the following classical \emph{Bernstein inequality}:
\begin{equation}\label{Eq:2-7}
\left\|D^{k}f\right\|_{L^{b}}
\leq C^{1+k} \lambda^{k+d\left(\frac{1}{a}-\frac{1}{b}\right)}\left\|f\right\|_{L^{a}}
\end{equation}
that holds for all function $f$ such that $\mathrm{Supp}\,\mathcal{F}f\subset\left\{\xi\in \R^{d}: \left|\xi\right|\leq R\lambda\right\}$ for some $R>0$
and $\lambda>0$, if $k\in\N$ and $1\leq a\leq b\leq\infty$.

More generally, if we assume $f$ to satisfy $\mathrm{Supp}\,\mathcal{F}f\subset \left\{\xi\in \R^{d}:
R_{1}\lambda\leq\left|\xi\right|\leq R_{2}\lambda\right\}$ for some $0<R_{1}<R_{2}$  and $\lambda>0$,
then for any smooth  homogeneous of degree $m$ function $A$ on $\R^d\setminus\{0\}$ and $1\leq a\leq\infty,$ we have
(see e.g. Lemma 2.2 in \cite{BCD}):
 \begin{equation}\label{Eq:2-8}
\left\|A(D)f\right\|_{L^{a}}\approx\lambda^{m}\left\|f\right\|_{L^{a}}.
\end{equation}
An obvious  consequence of \eqref{Eq:2-7} and \eqref{Eq:2-8} is that
$\left\|D^{k}f\right\|_{\dot{B}^{s}_{p, r}}\thickapprox\left \|f\right\|_{\dot{B}^{s+k}_{p, r}}$ for all $k\in\N.$

We also need the following nonlinear generalization of \eqref{Eq:2-8} (see Lemma 8 in \cite{DR3}):
\begin{prop}\label{Prop2.5}
If $\mathrm{Supp}\,\mathcal{F}f\subset \left\{\xi\in \R^{d}:
R_{1}\lambda\leq\left|\xi\right|\leq R_{2}\lambda\right\}$ then there exists $c$ depending only on $d,$ $R_1$  and $R_2$
so that for all $1<p<\infty$,
$$c\lambda^{2}\left(\frac{p-1}{p}\right)\int_{\R^{d}}\left|f\right|^{p}dx\leq (p-1)\int_{\R^d}\left|\nabla f\right|^{2}\left|f\right|^{p-2}dx
=-\int_{\R^{d}}\Delta f\left|f\right|^{p-2}f\,dx.
$$
\end{prop}
A time-dependent version of  the following  commutator estimate has been used  in the second step of the proof of Theorem \ref{Thm1.2}.
\begin{prop}\label{Prop2.6} (\cite{DX})
Let $1\leq p,\,p_{1}\leq\infty$ and
$$
-\min\left(\frac{d}{p_{1}},\frac{d}{p'}\right)<\sigma\leq1+\min\left(\frac {d}{p},\frac{d}{p_{1}}\right).
$$
There exists a constant $C>0$ depending only on $\sigma$ such that for all $k\in\Z$ and $\ell\in\left\{1,\cdots,d\right\}$, we have
$$
\left\|\left[v\cdot\nabla,\d_\ell\dot{\Delta}_{k}\right]a\right\|_{L^{p}}\leq
Cc_{k}2^{-k\left(\sigma-1\right)}\left\|\nabla v\right\|_{\dot{B}^{\frac{d}{p_{1}}}_{p_{1},1}}\left\|\nabla a\right\|_{\dot{B}^{\sigma-1}_{p,1}},
$$
where the commutator
$\left[\cdot,\cdot\right]$ is defined by $\left[f,g\right]=fg-gf$ and $(c_{k})_{k\in\Z}$ denotes
a sequence such that $\left\|(c_{k})\right\|_{\ell^{1}}\leq 1$.
\end{prop}
Let us finally recall the following parabolic regularity estimate for the
heat equation to end this section.
\begin{prop}\label{Prop2.7}
Let $\sigma\in \R$, $(p,r)\in \left[1,\infty\right]^{2}$ and $1\leq \rho_{2}\leq\rho_{1}\leq\infty$. Let $u$ satisfy
$$\left\{\begin{array}{lll}
\d_{t}u-\mu\Delta u=f,\\
u_{|t=0}=u_{0}.
\end{array}
\right.$$
Then for all $T>0$ the following a priori estimate is fulfilled:
\begin{equation}\label{Eq:2-9}\mu^{\frac1{\rho_1}}\left\|u\right\|_{\tilde L_{T}^{\rho_1}(\dot B^{\sigma+\frac{2}{\rho_1}}_{p,r})}\lesssim
\left\|u_{0}\right\|_{\dot {B}^{\sigma}_{p,r}}+\mu^{\frac{1}{\rho_{2}}-1}\left\|f\right\|_{\tilde L^{\rho_{2}}_{T}(\dot {B}^{\sigma-2+\frac{2}{\rho_{2}}}_{p,r})}.
\end{equation}
\end{prop}
\begin{rem} \label{Rem2.2}
The solutions to the following \emph{Lam\'e system}
\begin{equation}\label{Eq:2-10}
\left\{\begin{array}{lll}\d_tu-\mu\Delta u-\left(\lambda+\mu\right)\nabla \div u=f,\\
u_{|t=0}=u_{0},
\end{array}
\right.
\end{equation}
where $\lambda$ and $\mu$ are constant coefficients such that $\mu>0$ and $\lambda+2\mu>0,$
also fulfill \eqref{Eq:2-9} (up to the dependence w.r.t. the viscosity).
Indeed: both $\cP u$ and $\cQ u$ satisfy a heat equation,
as may be easily observed by applying $\cP$ and $\cQ$ to \eqref{Eq:2-10}.
\end{rem}

\section{The proof of time-decay estimates} \setcounter{equation}{0}
This section is devoted to the proofs of main results taking for granted the global-in-time existence result of Theorem \ref{Thm1.1}.
Regarding Theorem \ref{Thm1.2}, we shall follow from the fashion of joint work with Danchin \cite{DX} and proceed the proof in three steps, according to three terms of the
time-weighted functional $\mathcal{D}_{p}$. In the present paper, we deal with the strong coupling between the magnetic field and fluid dynamics. Precisely, the first step is dedicated to the proof of decay estimates for the low frequency part of $(a,u,H)$. By using the energy approaches for partially dissipative first-order symmetric systems, which was initiated by Godunov \cite{GSK} and developed by Kawashima et al. \cite{SK} (see also the survey paper \cite{DR2} due to Danchin), we construct the low frequency decay properties of semi-group defined by the left-hand side of system \eqref{Eq:1-6}, which enables us to pay less attention on the spectral analysis on the linearized MHD. Then we deduce the decay estimate of the first term of functional $\mathcal{D}_{p}$ with aid of the frequency-localization Duhamel principle. Having assumption \eqref{Eq:1-7}, we work on the $\dot{B}^{-s_{0}}_{2,\infty}$-$L^{2}$ estimates, actually, $L^{p/2}$-$L^{2}$ ones if $2\leq p\leq d$. However, if $p>d$ (a case that is relevant in physical dimension $d=2,3$), the low frequency part of nonlinear terms need not to be $L^{p/2}$. Consequently, some non so-classical product estimates in Besov spaces (see Propositions \ref{Prop2.2}-\ref{Prop2.3}) will be used.

In the second step, we are concerned with the time decay of the high frequencies part of the solution. In the case of higher Sobolev regularity, Duhamel principle always remains valid accepting loss of derivatives as the case may be. In our critical regularity framework however, one cannot afford any loss of regularity for the high frequency of the solution (for instance, like $u\cdot\nabla a$ induces a loss of one derivative as one cannot expect any smoothing for $a$, solution of a transport equation), so the usual Duhamel principle is now no longer valid. To get around the difficulty, we introduce the effective velocity (see for example \cite{HB})
\begin{equation}\label{Eq:3-1}
w\triangleq \nabla \left(-\Delta \right)^{-1}\left( a-\div u\right).
\end{equation}
Up to some low-order terms, we see that $H$, $w$ and the divergence free part $\cP u$ of $u$ all fulfill the ordinary heat equation, and
$a$ satisfies a damped transport equation. Thanks to the spectral-localization energy approach of $L^{p}$, we can achieve the desired high-frequency estimate for the second term of $\cD_{p}$.
In the last step, we establish gain of regularity and decay altogether for the high frequencies of the velocity and magnetic field, due to the parabolic smooth effect (see Proposition \ref{Prop2.7} and Remark \ref{Rem2.2}).

In what follows, we shall use repeatedly that for $0\leq s_{1}\leq s_{2}$ with $s_{2}>1$, we have
\begin{equation}\label{Eq:3-2}
\int_{0}^{t}\langle t-\tau\rangle^{-s_{1}}\langle\tau\rangle^{-s_{2}}d\tau\lesssim\langle t\rangle^{-s_{1}} ,
\end{equation}
\begin{equation}\label{Eq:3-3}
\int_{0}^{t}\langle t-\tau\rangle^{-s_{1}}\tau^{-\theta}\langle\tau\rangle^{\theta-s_{2}}d\tau\lesssim\langle t\rangle^{-s_{1}} \ \ \hbox{if} \ \ 0\leq\theta<1,
\end{equation}
and that the global solution $\left(a,u,H\right)$ given by Theorem \ref{Thm1.1} satisfies
\begin{equation}\label{Eq:3-4}
\left\|a\right\|_{\tilde{L}^{\infty}_{t}(\dot{B}^{\frac{d}{p}}_{p,1})}\leq c \ll 1 \ \ \hbox{for all} \ \ t\geq0.
\end{equation}

\subsection{First step: Bounds for the Low Frequencies}
Let $\left(\mathcal{G}(t)\right)_{t\geq0}$ be the semi-group associated with the left-hand side of system \eqref{Eq:1-6}. It follows from the
standard Duhamel formula that
\begin{equation}\label{Eq:3-5}
\left(
\begin{array}{c}
a\\
u\\
H
\end{array}
\right)
=\mathcal{G}(t)
\left(
\begin{array}{c}
a_{0}\\
u_{0}\\
H_{0}
\end{array}
\right)
+\int_{0}^{t}\mathcal{G}\left(t-\tau\right)
\left(
\begin{array}{c}
f(\tau)\\
g(\tau)\\
m(\tau)
\end{array}
\right)d\tau.
\end{equation}
The following Lemma indicates the low frequencies of $\left(a_{L},u_{L},H_{L}\right)\triangleq\mathcal{G}(t)\left(a_{0},u_{0},H_{0}\right)$
behaves essentially as the solution to the heat equation.
\begin{lem}\label{Lem3.1}
Let $\left(a_{L},u_{L},H_{L}\right)$ be the solution to the following system
\begin{equation*}
\left\{
\begin{array}{l}
\partial_{t}a_{L}+\div u_{L}=0, \\
\partial_{t}u_{L}-\cA u_{L}+\nabla a_{L}+\nabla\left(I\cdot H_{L}\right)-I\cdot\nabla H_{L}=0, \\
\partial_{t}H_{L}-\Delta H_{L}+\left(\div u_{L}\right)I-I\cdot\nabla u_{L}=0,\\
\div H_{L}=0,
\end{array}
\right.
\end{equation*}
with the initial data
\begin{equation*}
\left(a_{L},u_{L},H_{L}\right)|_{t=0}=\left(a_{0},u_{0},H_{0}\right).
\end{equation*}
Then, for any $k_{0}\in\Z$, there exists a positive constant $c_{0}=c_{0}\left(\lambda_{\infty},\mu_{\infty},k_{0}\right)$
such that
\begin{equation*}
\left\|\left(a_{L,k},u_{L,k},H_{L,k}\right)(t)\right\|_{L^{2}}\lesssim e^{-c_{0}2^{2k}t} \left\|\left(a_{0,k},u_{0,k},H_{0,k}\right)\right\|_{L^{2}}
\end{equation*}
for $t\geq0$ and $k\leq k_{0}$, where we set $z_{k}=\dot{\Delta}_{k}z$ for any $z\in S'(\R^{d})$.
\end{lem}
\begin{proof}
Let $\Lambda^{s}z\triangleq\mathcal{F}^{-1}\left(\left|\xi\right|^{s}\mathcal{F}z\right)$ ($s\in\R$),
$\omega_{L}=\Lambda^{-1}\mathrm{div}\,u_{L}$, $\Omega_{L}=\Lambda^{-1}\mathrm{curl}\,u_{L}$ and $E_{L}=\Lambda^{-1}\mathrm{curl}\,H_{L}$ where $\mathrm{curl}\,v=\left(\partial_{j}v_{i}-\partial_{i}v_{j}\right)_{ij}$ is a $d\times d$ matrix. Then we see that
\begin{equation}\label{Eq:3-6}
\left\{
\begin{array}{l}
\partial _{t}a_{L}+\Lambda \omega_{L}=0, \\
\partial _{t}\omega_{L}-\Delta \omega_{L}-\Lambda a_{L}-I\cdot \mathrm{div}\,E_{L}=0, \\
\partial _{t}\Omega_{L}-\mu_{\infty}\Delta \Omega_{L}-I\cdot \nabla E_{L}=0, \\
\partial _{t}E_{L}-\Delta E_{L}+\mathrm{curl}\,\left(\omega_{L}I\right)-I\cdot\nabla \Omega_{L}=0,\\
u_{L}=-\Lambda^{-1}\nabla\omega_{L}+\Lambda^{-1}\mathrm{div}\,\Omega_{L},\ \ \ H_{L}=\Lambda^{-1}\mathrm{div}\,E_{L}, \ \ \ \mathrm{div}\,H_{L}=0,
\end{array}
\right.
\end{equation}
where $\left(\mathrm{div}E\right)_{j}=\sum \limits_{i=1}^{d}\partial_{i}E_{ij}$ with entries $E_{ij}$ of the matrix $E$.

Set $\left(A,V,W,M\right)$ be the Fourier transform of $\left(a_{L},\omega_{L},\Omega_{L},E_{L}\right)$. Then we get from \eqref{Eq:3-6}
\begin{equation*}
\left\{
\begin{array}{l}
\partial _{t}A+\left|\xi\right|V=0, \\
\partial _{t}V+\left|\xi\right|^{2}V-\left|\xi\right|A-I\cdot \left(i\,\xi M\right) =0, \\
\partial _{t}W+\mu_{\infty}\left|\xi\right|^{2}W-i\left(I\cdot \xi\right)M=0, \\
\partial _{t}M+\left|\xi\right|^{2} M+i\,\Gamma\,V-i\left(I\cdot \xi\right)W=0,
\end{array}
\right.
\end{equation*}
where $\Gamma=\left(\xi_{j}I_{i}-\xi_{i}I_{j}\right)_{ij}$ is a $d\times d$ matrix. Hence, we easily get the following four identities
\begin{eqnarray}\label{Eq:3-7}
&&\frac{1}{2}\frac{d}{dt}\left|A\right|^{2}+
\left|\xi\right|\mathrm{Re}\left(A\bar{V}\right)=0,\\
\label{Eq:3-8}
&&\frac{1}{2}\frac{d}{dt}\left|V\right|^{2}+\left|\xi\right|^{2}\left|V\right|^{2}
-\left|\xi\right|\mathrm{Re}\left(A\bar{V}\right)
-\mathrm{Re}\left(\left(I\cdot \left(\mathrm{i}\,\xi M\right)\right)\bar{V}\right)=0,\\
\label{Eq:3-9}
&&\frac{1}{2}\frac{d}{dt}\left|W\right|^{2}+\mu_{\infty}|\xi|^{2}\left|W\right|^{2}
-\left(I\cdot \xi\right)\mathrm{Re}\left\langle \mathrm{i}\,M,W\right\rangle= 0,\\
\label{Eq:3-10}
&&\frac{1}{2}\frac{d}{dt}\left|M\right|^{2}+\left|\xi\right|^{2}\left|M\right|^{2}
+\mathrm{Re}\left\langle \mathrm{i}\,\Gamma\,V,M\right\rangle
-\left(I\cdot \xi\right)\mathrm{Re}\left\langle \mathrm{i}\,W,M\right\rangle=0,
\end{eqnarray}
where $\bar{f}$ indicates the complex conjugate of a function $f$.
Noticing that
\begin{eqnarray*}
&&\hspace{-5mm}\mathrm{Re}\left\langle i\,\Gamma V,M\right\rangle
=\mathrm{Re}\left(\sum^{d}_{k,j=1}i\left(\xi_{j}I_{k}-\xi_{k}I_{j}\right)V \bar{M}_{kj}\right)
=-2\,\mathrm{Re}\left(\sum^{d}_{k,j=1}i\,\xi_{k}I_{j}V\bar{M}_{kj}\right)\\
&&\hspace{1.7cm}=2\,\mathrm{Re}\left(\sum^{d}_{k,j=1}i\,\xi_{k}I_{j}\bar{V}M_{kj}\right)
=2\,\mathrm{Re}\left(\left(I\cdot \left(i\,\xi M\right)\right)\bar{V}\right)
\end{eqnarray*}
and
\begin{equation*}
\mathrm{Re}\left\langle i\,M,W\right\rangle=-\mathrm{Re}\left\langle i\,W,M\right\rangle.
\end{equation*}
By adding up \eqref{Eq:3-7}, \eqref{Eq:3-8}, \eqref{Eq:3-9} and \eqref{Eq:3-10}, we arrive at
\begin{equation}\label{Eq:3-11}
\frac{1}{2}\frac{d}{dt}\left|\left(A,V,W,M\right)\right|^{2}
+\bar{\mu}_{\infty}\left|\xi\right|^{2}\left|\left(V,W,M\right)\right|^{2}\leq 0 \ \text{ with } \ \bar{\mu}_{\infty}\triangleq \min\left(1,\mu_{\infty}\right).
\end{equation}
Next, we need the following equality to create the dissipation of $A$:
\begin{equation*}
\frac{d}{dt}\left[-\mathrm{Re}\,\left(A\bar{V}\right)\right]
+\left|\xi\right|\left|A\right|^{2}-\left|\xi\right|\left|V\right|^{2}
-\left|\xi\right|^{2}\mathrm{Re} \left(A\bar{V}\right)
+\mathrm{Re}\left(\left(I\cdot \left(i\,\xi M\right)\right)\bar{A}\right)=0.
\end{equation*}
Together with \eqref{Eq:3-7}, furthermore, we infer that
\begin{eqnarray}
 \frac{1}{2}\frac{d}{dt}\left[\left|\left|\xi\right|A\right|^{2}-2\left|\xi\right|\mathrm{Re}\left(A\bar{V}\right)\right]
+\left|\xi\right|^{2}\left|A\right|^{2}-\left|\xi\right|^{2}\left|V\right|^{2} \nonumber \\
\label{Eq:3-12}
+\left|\xi\right|\mathrm{Re}\left(\left(I\cdot\left(i\,\xi M\right)\right)\bar{A}\right)= 0.
\end{eqnarray}
Hence, by introducing the ``Lyapunov functional"
$$\mathcal{L}^{2}(t)\triangleq \left|\left(A,V,W,M\right)\right|^{2}
+\sigma\left(\left|\xi\right|^{2}\left|A\right|^{2}-2\left|\xi\right|\mathrm{Re}\left(A\bar{V}\right)\right)$$
for some $\sigma>0$, we get from \eqref{Eq:3-11} and \eqref{Eq:3-12}
\begin{eqnarray*}
\frac{1}{2}\frac{d}{dt}\mathcal{L}^{2}(t)
+|\xi|^{2}\left(\sigma\left|A\right|^{2}+\left(\bar{\mu}_{\infty}-\sigma\right)\left|V\right|^{2}
+\bar{\mu}_{\infty}\left|\left(W,M\right)\right|^{2}\right)\nonumber\\
+\sigma\left|\xi\right|\mathrm{Re}\left(\left(I\cdot\left(i\,\xi M\right)\right)\bar{A}\right) \leq 0.
\end{eqnarray*}
Then, taking advantage of the basic inequality
$$\left|\mathrm{Re}\left(\left(I\cdot\left(i\,\xi M\right)\right)\bar{A}\right)\right|
\leq \frac12 \left|\xi\right|(\left|A\right|^{2}+\left|M\right|^{2}),$$
and then choosing the constant $\sigma$ such that $\sigma=\frac{1}{2}\,\bar{\mu}_{\infty}$, we end up with
\begin{equation}\label{Eq:3-13}
\frac{d}{dt}\mathcal{L}^{2}(t)
+\bar{\mu}_{\infty}\left|\xi\right|^{2}\left(\frac{1}{2}\left|A\right|^{2}+\left|V\right|^{2}
+2\left|W\right|^{2}+\frac{3}{2}\left|M\right|^{2}\right) \leq 0.
\end{equation}
On the other hand, it is not difficult to conclude that there exist some constant $C_{0}>0$ depending only on $\mu_{\infty}$ so that
\begin{equation}\label{Eq:3-14}
C^{-1}_{0}\mathcal{L}^{2}(t)\leq \left|\left(A,\left|\xi\right|A,V,W,M\right)\right|^{2}\leq C_{0}\mathcal{L}^{2}(t).
\end{equation}
This in particular implies that for arbitrary fixed $\rho_{0}\triangleq \left|\xi_{0}\right|>0$ we get for some constant $c_{0}$ depending only on $\rho_{0}$ and $\bar{\mu}_{\infty}$,
$$\bar{\mu}_{\infty}\left(\frac{1}{2}\left|A\right|^{2}+\left|V\right|^{2}
+2\left|W\right|^{2}+\frac{3}{2}\left|M\right|^{2}\right)\geq c_{0}\mathcal{L}^{2}(t)\ \ \ \mbox{for all}\ \ \ 0\leq |\xi|\leq\rho_{0}.$$
Hence, reverting to \eqref{Eq:3-13} yields
$$\mathcal{L}^{2}(t)\leq e^{-c_{0}\left|\xi\right|^{2}t}\mathcal{L}^{2}(0).$$
Then we get thanks to \eqref{Eq:3-14},
\begin{equation*}
\left|\left(A,V,W,M\right)(t)\right|
\leq C e^{-\frac{c_{0}}{2}|\xi|^{2}t}\left|\left(A,V,W,M\right)(0)\right|
\end{equation*}
for all $t\geq0 $ and $ 0\leq |\xi|\leq\rho_{0}$. Finally, multiplying both sides of the above inequality by $\varphi\left(2^{-k}\xi\right)$,
with the aid of Fourier-Plancherel theorem, we end up for all $k\leq k_{0}$ with
$$ \left\|\left(a_{L,k},\omega_{L,k},\Omega_{L,k},E_{L,k}\right)(t)\right\|_{L^{2}}
\lesssim e^{-\frac{c_{0}}{2}2^{2k}t}\left\|\left(a_{L,k},\omega_{L,k},\Omega_{L,k},E_{L,k}\right)(0)\right\|_{L^{2}},
$$
which leads to our desired inequality eventually. The proof of Lemma \ref{Lem3.1} is complete.
\end{proof}
Set $U\triangleq\left(a,u,H\right)$ and $U_{0}\triangleq\left(a_{0},u_{0},H_{0}\right)$. Based on the key lemma
(Lemma \ref{Lem3.1}), we perform the same procedure as in \cite{DX} to get for $s+s_{0}>0$
\begin{equation*}
\sup_{t\geq 0}t^{\frac{s_{0}+s}{2}}\left\|\mathcal{G}(t)U\right\|^{\ell}_{\dot{B}^{s}_{2,1}}
\lesssim \left\|U\right\|^{\ell}_{\dot{B}^{-s_{0}}_{2,\infty}}.
\end{equation*}
Additionally, it is clear that for $s+s_{0}>0$,
\begin{equation*}
\left\|\mathcal{G}(t)U\right\|^{\ell}_{\dot{B}^{s}_{2,1}}\lesssim \left\|U\right\|^{\ell}_{\dot{B}^{-s_{0}}_{2,\infty}}\sum_{k\leq k_{0}}2^{k\left(s_{0}+s\right)}\lesssim \left\|U\right\|^{\ell}_{\dot{B}^{-s_{0}}_{2,\infty}}.
\end{equation*}
Hence, setting $\langle t\rangle\triangleq\sqrt{1+t^{2}}$, we obtain
\begin{equation}\label{Eq:3-15}
\sup_{t\geq 0}\langle t\rangle^{\frac{s_{0}+s}{2}}\left\|\mathcal{G}(t)U_{0}\right\|^{\ell}_{\dot{B}^{s}_{2,1}}\lesssim \left\|U_{0}\right\|^{\ell}_{\dot{B}^{-s_{0}}_{2,\infty}}.
\end{equation}
Furthermore, it follows from Duhamel's formula that
\begin{equation*}
\left\|\int_{0}^{t}\mathcal{G}(t-\tau)\left(f,g,m\right)(\tau)d\tau\right\|^{\ell}_{\dot{B}^{s}_{2,1}}\lesssim \int_{0}^{t}\langle t-\tau\rangle^{-\frac{s_{0}+s}{2}} \left\|\left(f,g,m\right)(\tau)\right\|^{\ell}_{\dot{B}^{-s_{0}}_{2,\infty}}d\tau.
\end{equation*}
We claim that if $p$ satisfies \eqref{Eq:1-4}, then we have for all $t\geq0$,
\begin{equation}\label{Eq:3-16}
\int_{0}^{t}\langle t-\tau\rangle^{-\frac{s_{0}+s}{2}} \left\|\left(f,g,m\right)(\tau)\right\|^{\ell}_{\dot{B}^{-s_{0}}_{2,\infty}}d\tau
\lesssim \langle t\rangle^{-\frac{s_{0}+s}{2}}\left(\mathcal{D}^{2}_{p}(t)+\mathcal{E}^{2}_{p}(t)\right),
\end{equation}
where $\mathcal{E}_{p}(t)$ and $\mathcal{D}_{p}(t)$ have been defined by \eqref{Eq:1-5} and \eqref{Eq:1-9}, respectively.

Indeed, we decompose the nonlinear term $g=\sum \limits_ {j=1}^{6}g_{j}$ with
\begin{eqnarray*}
&&g_{1}=-u\cdot \nabla u,\ \ \ g_{2}=-\pi_{2}(a) \nabla a, \\
&&g_{3}=\frac {1}{1+a}\left( 2\widetilde{\mu }(a)\,\mathrm{div}\,D(u)+%
\widetilde{\lambda}(a)\,\nabla \mathrm{div}\,u\right) -\pi_{1}(a)\mathcal{A}u, \\
&&g_{4}=\frac {1}{1+a}\left( 2\widetilde{\mu}'(a)\,\mathrm{div}\,D(u)\cdot\nabla a+\widetilde{\lambda}'(a)\,\mathrm{div}\,u\, \nabla
a\right),\\
&&g_{5}=\pi_{1}(a)\left(\nabla\left(I\cdot H\right)-I\cdot\nabla H\right), \ \ \ g_{6}=-\frac {1}{1+a}\left(\frac{1}{2}\nabla\left|H\right|^{2}-H\cdot\nabla H\right).
\end{eqnarray*}
As shown by \cite{DX}, we can get the following inequality
\begin{equation*}
\int_{0}^{t}\langle t-\tau\rangle^{-\frac{s_{0}+s}{2}}\left\|\left(f,g_{1},g_{2},g_{3},g_{4}\right)\right\|^{\ell}_{\dot{B}^{-s_{0}}_{2,\infty}}d\tau\lesssim\langle t\rangle^{-\frac{s_{0}+s}{2}}\left(\mathcal{E}_{p}^{2}(t)+\mathcal{D}^{2}_{p}(t)\right).
\end{equation*}
In order to finish the proof of \eqref{Eq:3-16}, it suffices to bound those terms coupled with magnetic field. To the end, it is convenient to decompose them
in terms of low-frequency and high-frequency as follows:
\begin{eqnarray*}
&&\hspace{-6mm}g_{5}(a,H)=\pi_{1}(a)\left( \nabla(I\cdot H^{\ell})-I\cdot\nabla H^{\ell}\right)+\pi_{1}(a)\left(\nabla(I\cdot H^{h})-I\cdot\nabla H^{h}\right),\\
&&\hspace{-6mm}g_{6}\left(a,H\right)=-\frac {1}{1+a}\left(\left(\nabla H^{\ell}\cdot H\right)^{T}-H\cdot\nabla H^{\ell}\right)
-\frac {1}{1+a}\left(\left(\nabla H^{h}\cdot H\right)^{T}-H\cdot\nabla H^{h}\right),\\
&&\hspace{-6mm}m(u,H)=-H(\mathrm{div}\,u^{\ell})+H\cdot\nabla u^{\ell}-u\cdot\nabla H^{\ell}-H(\mathrm{div}\,u^{h})+H\cdot\nabla u^{h}-u\cdot\nabla H^{h},
\end{eqnarray*}
where
$$z^{\ell}\triangleq\sum_{k\leq k_0}\ddk z\ \ \hbox{and}\ \ z^{h}\triangleq z-z^{\ell} \quad\hbox{for}\quad z=u,H.$$

In addition, we shall use repeatedly the following Moser's product inequalities in Besov spaces with negative regularity, which are the direct consequences of Proposition \ref{Prop2.2}.
\begin{lem}
If $p$ satisfies \eqref{Eq:1-4}, then it holds that
\begin{eqnarray}\label{Eq:3-17}
&&\left\|FG\right\|_{\dot {B}^{-s_{0}}_{2,\infty}}
\lesssim \left\|F\right\|_{\dot {B}^{1-\frac {d}{p}}_{p,1}}\left\|G\right\|_{\dot {B}^{\frac {d}{2}-1}_{2,1}},\\
\label{Eq:3-18}
&&\left\|FG\right\|_{\dot {B}^{-\frac {d}{p}}_{2,\infty}}
\lesssim \left\|F\right\|_{\dot {B}^{\frac {d}{p}-1}_{p,1}}\left\|G\right\|_{\dot{B}^{1-\frac {d}{p}}_{2,1}},\\
\label{Eq:3-19}
&&\left\|FG\right\|_{\dot {B}^{1-\frac {d}{p}}_{p,1}}\lesssim \left\|F\right\|_{\dot {B}^{\frac {d}{p}}_{p,1}}\left\|F\right\|_{\dot {B}^{1-\frac {d}{p}}_{p,1}}.
\end{eqnarray}
\end{lem}
The interested reader is referred to \cite{DX} for the proofs of \eqref{Eq:3-17}-\eqref{Eq:3-18}. By taking $\sigma_{2}=1-\frac {d}{p}, \sigma_{1}=\frac {d}{p}$ and $p_{1}=p_{2}=p$, we can get \eqref{Eq:3-19} readily. Owing to those embedding properties and the definition of $\mathcal{D}_p$, we arrive at
\begin{eqnarray} \label{Eq:3-20}
&&\left\|\left(a,u,H\right)^{\ell}(\tau)\right\|_{\dot{B}^{1-\frac{d}{p}}_{p,1}}\lesssim \left\|\left(a,u,H\right)^{\ell}(\tau)\right\|_{\dot{B}^{1-s_{0}}_{2,1}}\lesssim\langle\tau\rangle^{-\frac{1}{2}}\mathcal{D}_{p}(\tau),\\
\label{Eq:3-21}
&&\left\|\left(a,u,H\right)^{\ell}(\tau)\right\|_{\dot{B}^{\frac{d}{p}}_{p,1}}\lesssim\left\|(a,u,H)^{\ell}(\tau)\right\|_{\dot{B}^{\frac{d}{2}}_{2,1}}
\lesssim\langle\tau\rangle^{-\frac{d}{p}}\mathcal{D}_{p}(\tau).
\end{eqnarray}
Due to the fact $p<2d$, it follows that
\begin{equation} \label{Eq:3-22}
\left\|a^{h}(\tau)\right\|_{\dot{B}^{1-\frac{d}{p}}_{p,1}} \lesssim \left\|a^{h}(\tau)\right\|_{\dot{B}^{\frac{d}{p}}_{p,1}}\lesssim
\langle\tau\rangle^{-\frac{d}{p}}\mathcal{D}_{p}(\tau) \lesssim \langle\tau\rangle^{-\frac{1}{2}}\mathcal{D}_{p}(\tau)
\ \ \hbox{for  all} \ \ \tau\geq0.
\end{equation}
Using the interpolation inequality in Besov spaces, we have
\begin{equation*}
\left\|(u,H)^{h}\right\|_{\dot{B}^{1-\frac{d}{p}}_{p,1}}\lesssim\left\|(u,H)^{h}\right\|_{\dot{B}^{\frac{d}{p}}_{p,1}}
\lesssim \left(\left\|(u,H)\right\|^{h}_{\dot{B}^{\frac{d}{p}-1}_{p,1}}\left\|\nabla (u,H)\right\|^{h}_{\dot{B}^{\frac{d}{p}}_{p,1}}\right)^{\frac{1}{2}},
\end{equation*}
and thus, according to the definition of $\mathcal{D}_{p}$,
\begin{equation}\label{Eq:3-23}
\left\|(u,H)^{h}(\tau)\right\|_{\dot{B}^{1-\frac{d}{p}}_{p,1}}\lesssim\left\|(u,H)^{h}\right\|_{\dot{B}^{\frac{d}{p}}_{p,1}}
\lesssim  \tau^{-\frac{1}{2}} \langle\tau\rangle^{-\frac{\alpha}{2}}\mathcal{D}_{p}(\tau) \ \ \hbox{for  all} \ \ \tau\geq0.
\end{equation}
Now, let us bound those terms of  $g_{5}$, $g_{6}$ and $m$ containing $u^{\ell}$ or $H^{\ell}$. By using \eqref{Eq:3-17} together with Proposition \ref{Prop2.4} and \eqref{Eq:3-4}, \eqref{Eq:3-20}-\eqref{Eq:3-22}, we arrive at
\begin{eqnarray*}
&&\int_{0}^{t}\langle t-\tau\rangle^{-\frac{s_{0}+s}{2}}\left\|\pi_{1}(a)\left(\nabla\left(I\cdot H^{\ell}\right)-I\cdot\nabla H^{\ell}\right)\right\|^{\ell}_{\dot{B}^{-s_{0}}_{2,\infty}}d\tau\\
&\lesssim &\int_{0}^{t}\langle t-\tau\rangle^{-\frac{s_{0}+s}{2}}
\left\|a\right\|_{\dot{B}^{1-\frac{d}{p}}_{p,1}}\left\| H^{\ell}\right\|_{\dot{B}^{\frac{d}{2}}_{2,1}}d\tau\\
&\lesssim &\mathcal{D}^{2}_{p}(t)\int_{0}^{t}\langle t-\tau\rangle^{-\frac{s_{0}+s}{2}}
\langle \tau\rangle^{-\left(\frac{d}{p}+\frac{1}{2}\right)} d\tau.
\end{eqnarray*}
Because $\frac{d}{p}+\frac{1}{2}>1$ and $\frac{s_{0}+s}{2}\leq\frac{d}{p}+\frac{1}{2}$ for $s\leq\frac{d}{2}+1$, inequality \eqref{Eq:3-2} implies that
\begin{equation*}
\int_{0}^{t}\langle t-\tau\rangle^{-\frac{s_{0}+s}{2}}\left\|\pi_{1}(a)\left(\nabla\left(I\cdot H^{\ell}\right)-I\cdot\nabla H^{\ell}\right)\right\|^{\ell}_{\dot{B}^{-s_{0}}_{2,\infty}}d\tau
\lesssim\langle t\rangle^{-\frac{s_{0}+s}{2}}\mathcal{D}^{2}_{p}(t).
\end{equation*}
Regarding the term with $-\frac {1}{1+a}\left(\left(\nabla H^{\ell}\cdot H\right)^{T}-H\cdot\nabla H^{\ell}\right)$, it follows from \eqref{Eq:3-4}, \eqref{Eq:3-17}, \eqref{Eq:3-19}-\eqref{Eq:3-21} and \eqref{Eq:3-23}, and Proposition \ref{Prop2.4} that
\begin{eqnarray*}
&&\int_{0}^{t}\langle t-\tau\rangle^{-\frac{s_{0}+s}{2}}\left\|\frac {1}{1+a}\left(\left(\nabla H^{\ell}\cdot H\right)^{T}-H\cdot\nabla H^{\ell}\right)\right\|^{\ell}_{\dot{B}^{-s_{0}}_{2,\infty}}d\tau\\
&\lesssim &\int_{0}^{t}\langle t-\tau\rangle^{-\frac{s_{0}+s}{2}}(1+\left\|a\right\|_{\dot{B}^{\frac{d}{p}}_{p,1}})
\left\|H\right\|_{\dot{B}^{1-\frac{d}{p}}_{p,1}}\left\| H^{\ell}\right\|_{\dot{B}^{\frac{d}{2}}_{2,1}}d\tau\\
&\lesssim &\int_{0}^{t}\langle t-\tau\rangle^{-\frac{s_{0}+s}{2}}
\left\|H\right\|_{\dot{B}^{1-\frac{d}{p}}_{p,1}}\left\| H^{\ell}\right\|_{\dot{B}^{\frac{d}{2}}_{2,1}}d\tau\\
&\lesssim &\mathcal{D}^{2}_{p}(t)\int_{0}^{t}\langle t-\tau\rangle^{-\frac{s_{0}+s}{2}}
\left(\langle \tau\rangle^{-\frac{1}{2}}+\tau^{-\frac{1}{2}}\langle \tau\rangle^{-\frac{\alpha}{2}} \right)\langle \tau\rangle^{-\frac{d}{p}} d\tau.
\end{eqnarray*}
Thanks to \eqref{Eq:3-2} and \eqref{Eq:3-3}, we end up with
\begin{equation*}
\int_{0}^{t}\langle t-\tau\rangle^{-\frac{s_{0}+s}{2}}
\left\|\frac {1}{1+a}\left(\left(\nabla H^{\ell}\cdot H\right)^{T}-H\cdot\nabla H^{\ell}\right)\right\|^{\ell}_{\dot{B}^{-s_{0}}_{2,\infty}}d\tau
\lesssim \langle t\rangle^{-\frac{s_{0}+s}{2}}\mathcal{D}^{2}_{p}(t).
\end{equation*}
The terms $-H\left(\mathrm{div}\,u^{\ell}\right)$, $H\cdot \nabla u^{\ell}$ and $-u\cdot \nabla H^{\ell}$ may be treated at a similar way (use \eqref{Eq:3-17}, \eqref{Eq:3-20}, \eqref{Eq:3-21} and \eqref{Eq:3-23}), so we feel free to skip them for brevity.

For the terms of $g_{5}$, $g_{6}$ and $m$ containing $u^{h}$ or $H^{h}$, we shall proceed differently depending on whether $p\leq d$ or  $p>d$.
Let us first consider the easier case $2\leq p\leq d$. To bound the term corresponding to $\pi_{1}(a)\left(\nabla\left(I\cdot H^{h}\right)-I\cdot\nabla H^{h}\right)$, we use the Sobolev embedding $L^{\frac{p}{2}}\hookrightarrow\dot{B}^{-s_{0}}_{2,\infty}$ and H\"{o}lder inequality, and get
\begin{eqnarray*}
&&\int_{0}^{t}\langle t-\tau\rangle^{-\frac{s_{0}+s}{2}}\left\|\pi_{1}(a)\left(\nabla\left(I\cdot H^{h}\right)-I\cdot\nabla H^{h}\right)\right\|^{\ell}_{\dot{B}^{-s_{0}}_{2,\infty}}d\tau\\
&\lesssim & \int_{0}^{t}\langle t-\tau\rangle^{-\frac{s_{0}+s}{2}}\left\|a\right\|_{L^{p}}\left\|\nabla H^{h}\right\|_{L^{p}}d\tau.
\end{eqnarray*}
Using embeddings $\dot{B}^{\frac{d}{2}-\frac{d}{p}}_{2,1}\hookrightarrow L^{p}$ and $\dot{B}^{0}_{p,1}\hookrightarrow L^{p}$, we easily obtain
\begin{equation*}
\left\|\left(a,u,H\right)\right\|_{L^{p}}\leq\left \|\left(a,u,H\right)^{\ell}\right\|_{L^{p}}+\left\|\left(a,u,H\right)^{h}\right\|_{L^{p}}
\leq \left\|\left(a,u,H\right)\right\|^{\ell}_{\dot{B}^{\frac{d}{2}-\frac{d}{p}}_{2,1}}+\left\|\left(a,u,H\right)\right\|^{h}_{\dot{B}^{\frac{d}{p}-1}_{p,1}}
\end{equation*}
for $p\leq d$. Furthermore, it follows from the definition of $\mathcal{D}_{p}(t)$ and the fact $\alpha\geq\frac{d}{2p}$ that
\begin{equation}\label{Eq:3-24}
\left\|\left(a,u,H\right)(\tau)\right\|_{L^{p}}\lesssim \langle\tau\rangle^{-\frac{d}{2p}}\mathcal{D}_{p}(\tau).
\end{equation}
Additionally, arguing as for proving \eqref{Eq:3-23}, we have for $\frac{d}{p}\geq1$,
\begin{equation}\label{Eq:3-25}
\left\|\nabla\left(u,H\right)^{h}(\tau)\right\|_{L^{p}}\lesssim\left\|\left(u,H\right)^{h}(\tau)\right\|_{\dot{B}^{\frac{d}{p}}_{p,1}}
\lesssim\tau^{-\frac{1}{2}}\langle\tau\rangle^{-\frac{\alpha}{2}}\mathcal{D}_{p}(\tau).
\end{equation}
Using \eqref{Eq:3-24}, \eqref{Eq:3-25} and the fact that $\frac{\alpha+1}{2}+\frac{d}{2p}>\frac dp + \frac12>1$ for all $s\leq\frac{d}{2}+1$, we conclude thanks to \eqref{Eq:3-3} that
\begin{eqnarray*}
&&\int_{0}^{t}\langle t-\tau\rangle^{-\frac{s_{0}+s}{2}}\left\|\pi_{1}(a)\left(\nabla\left(I\cdot H^{h}\right)-I\cdot\nabla H^{h}\right)\right\|^{\ell}_{\dot{B}^{-s_{0}}_{2,\infty}}d\tau\\
&\lesssim & \mathcal{D}^{2}_{p}(t)\int_{0}^{t}\langle t-\tau\rangle^{-\frac{s_{0}+s}{2}}\tau^{-\frac{1}{2}}\langle \tau\rangle^{-\frac{d}{2p}-\frac{\alpha}{2}}d\tau\\
&\lesssim& \mathcal \langle t\rangle^{-\frac{s_{0}+s}{2}} {D}^{2}_{p}(t).
\end{eqnarray*}
Note that the embedding $\dot{B}^{\frac{d}{p}}_{p,1}\hookrightarrow L^{\infty}$, \eqref{Eq:3-3}, \eqref{Eq:3-4}, \eqref{Eq:3-24} and \eqref{Eq:3-25}, we obtain
\begin{eqnarray*}
&&\int_{0}^{t}\langle t-\tau\rangle^{-\frac{s_{0}+s}{2}}
\left\|\frac {1}{1+a}\left(\left(\nabla H^{h}\cdot H\right)^{T}-H\cdot\nabla H^{h}\right)\right\|^{\ell}_{\dot{B}^{-s_{0}}_{2,\infty}}d\tau\\
&\lesssim & \int_{0}^{t}\langle t-\tau\rangle^{-\frac{s_{0}+s}{2}}(1+\left\|a\right\|_{\dot{B}^{\frac{d}{p}}_{p,1}})\left\|H\right\|_{L^{p}}\left\|\nabla H^{h}\right\|_{L^{p}}d\tau\\ &\lesssim & \mathcal{D}^{2}_{p}(t)\int_{0}^{t}\langle t-\tau\rangle^{-\frac{s_{0}+s}{2}}\tau^{-\frac{1}{2}}\langle \tau\rangle^{-\frac{d}{2p}-\frac{\alpha}{2}}d\tau\\
&\lesssim &  \langle t\rangle^{-\frac{s_{0}+s}{2}} \mathcal{D}^{2}_{p}(t).
\end{eqnarray*}
Bounding $-H(\mathrm{div}\,u^{h}), \ \ H\cdot \nabla u^{h}$ and $-u\cdot \nabla H^{h}$ essentially follows from the same procedure, we thus omit them.

Let us end that step by considering the terms of $g_{5}\left(a,H\right)$, $g_{6}\left(a,H\right)$ and $m\left(u,H\right)$ with $u^{h}$ or $H^{h}$ in the case of $p>d$. We observe that applying \eqref{Eq:2-5} with $\sigma=1-\frac{d}{p}$ yields
\begin{equation}\label{Eq:3-26}
\|F G^{h}\|^{\ell}_{\dot{B}^{-s_{0}}_{2,\infty}}\lesssim (\|F\|_{\dot{B}^{1-\frac{d}{p}}_{p,1}}+\|\dot{S}_{k_{0}+N_{0}}F\|_{L^{p^{*}}})\|G^{h}\|_{\dot{B}^{\frac{d}{p}-1}_{p,1}}
\ \ \hbox{with} \ \ \frac{1}{p^{*}}\triangleq\frac{1}{2}-\frac{1}{p},
\end{equation}
and thus, because $\dot{B}^{\frac{d}{p}}_{2,1}\hookrightarrow L^{p^{*}}$,
\begin{equation}\label{Eq:3-27}
\|F G^{h}\|^{\ell}_{\dot{B}^{-s_{0}}_{2,\infty}}\lesssim (\|F^{\ell}\|_{\dot{B}^{\frac{d}{p}}_{2,1}}
+\|F\|_{\dot{B}^{1-\frac{d}{p}}_{p,1}})\|G^{h}\|_{\dot{B}^{\frac{d}{p}-1}_{p,1}}.
\end{equation}
To bound the term with $\pi_{1}(a)\left(\nabla\left(I\cdot H^{h}\right)-I\cdot \nabla H^{h}\right)$,
we see that, using the composition inequality in Lebesgue spaces and the embeddings $\dot{B}^{\frac{d}{p}}_{2,1}\hookrightarrow L^{p^{*}}$ and $\dot{B}^{s_{0}}_{p,1}\hookrightarrow L^{p^{*}}$, we get
$$
\left\|\pi_{1}(a)\right\|_{L^{p^{*}}}\lesssim\left\|a\right\|_{L^{p^{*}}}\lesssim\left\|a^{\ell}\right\|_{\dot{B}^{\frac{d}{p}}_{2,1}}+\left\|a^{h}\right\|_{\dot{B}^{s_{0}}_{p,1}}
\lesssim\left\|a^{\ell}\right\|_{\dot{B}^{\frac{d}{p}}_{2,1}}+\left\|a^{h}\right\|_{\dot{B}^{\frac{d}{p}}_{p,1}}.
$$
Hence, taking advantage of \eqref{Eq:3-4}, \eqref{Eq:3-26} and Proposition \ref{Prop2.4}, we have
\begin{eqnarray*}
&&\int_{0}^{t}\langle t-\tau\rangle^{-\frac{s_{0}+s}{2}}\left\|\pi_{1}(a)\left(\nabla(I\cdot H^{h})-I\cdot\nabla H^{h}\right)\right\|^{\ell}_{\dot{B}^{-s_{0}}_{2,\infty}}d\tau\\
&\lesssim &\int_{0}^{t}\langle t-\tau\rangle^{-\frac{s_{0}+s}{2}}(\|a^{\ell}\|_{\dot{B}^{\frac{d}{p}}_{2,1}}+\|a^{\ell}\|_{\dot{B}^{1-\frac{d}{p}}_{p,1}}
+\|a^{h}\|_{\dot{B}^{\frac{d}{p}}_{p,1}})\|\nabla H^{h}\|_{\dot{B}^{\frac{d}{p}-1}_{p,1}}d\tau.
\end{eqnarray*}
It follows from the definition of $\mathcal{D}_{p}(t)$ that
\begin{equation} \label{Eq:3-28}
\left\|\left(a,u,H\right)^{\ell}(\tau)\right\|_{\dot{B}^{\frac{d}{p}}_{2,1}}\lesssim \langle\tau\rangle^{-\left(\frac{3d}{2p}-\frac{d}{4}\right)}\mathcal{D}_{p}(\tau).
\end{equation}
Hence, with the aid of \eqref{Eq:3-20}, \eqref{Eq:3-22}, \eqref{Eq:3-23} and \eqref{Eq:3-28}, we conclude that
\begin{eqnarray*}
&&\int_{0}^{t}\langle t-\tau\rangle^{-\frac{s_{0}+s}{2}}\left\|\pi_{1}(a)\left(\nabla\left(I\cdot H^{h}\right)-I\cdot\nabla H^{h}\right)\right\|^{\ell}_{\dot{B}^{-s_{0}}_{2,\infty}}d\tau\\
&\lesssim & \mathcal{D}^{2}_{p}(t)\int_{0}^{t}\langle t-\tau\rangle^{-\frac{s_{0}+s}{2}}\tau^{-\frac{1}{2}}\langle\tau\rangle^{-\min\left(\frac{3d}{2p}-\frac{d}{4},\frac{1}{2}\right)-\frac{\alpha}{2}}d\tau.
\end{eqnarray*}
Due to the fact that $\frac{s_{0}+s}{2}\leq \frac{\alpha+1}{2}+\min\left(\frac{3d}{2p}-\frac{d}{4},\frac{1}{2}\right)$ for $p>d$ satisfying \eqref{Eq:1-4} and $s\leq\frac{d}{2}+1$, we obtain (thanks to \eqref{Eq:3-3})
\begin{equation*}
\int_{0}^{t}\langle t-\tau\rangle^{-\frac{s_{0}+s}{2}}\left\|\pi_{1}(a)\left(\nabla\left(I\cdot H^{h}\right)-I\cdot\nabla H^{h}\right)\right\|^{\ell}_{\dot{B}^{-s_{0}}_{2,\infty}}d\tau
\lesssim \langle t\rangle^{-\frac{s_{0}+s}{2}}\mathcal{D}^{2}_{p}(t).
\end{equation*}
Next, with aid of \eqref{Eq:3-19}, Propositions \ref{Prop2.2} and \ref{Prop2.4} and the embedding $\dot{B}^{s_{0}}_{p,1}\hookrightarrow L^{p^{*}}$, we arrive at
\begin{eqnarray*}
\left\|\frac {1}{1+a}H\right\|_{\dot{B}^{1-\frac{d}{p}}_{p,1}}
&\lesssim&\left(1+\left\|a\right\|_{\dot{B}^{\frac{d}{p}}_{p,1}}\right)
\left(\left\|H^{\ell}\right\|_{\dot{B}^{1-\frac{d}{p}}_{p,1}}+\left\|H^{h}\right\|_{\dot{B}^{1-\frac{d}{p}}_{p,1}}\right),\\
\left\|\dot{S}_{k_{0}+N_{0}}\left(\frac {1}{1+a}H\right)\right\|_{L^{p^{*}}}
&\lesssim&\left\|\dot{S}_{k_{0}+N_{0}}\left(\frac {1}{1+a}H\right)\right\|_{\dot{B}^{s_{0}}_{p,1}}
\lesssim\left\|\frac {1}{1+a}H\right\|_{\dot{B}^{\frac{d}{p}-1}_{p,1}}\\
&\lesssim&\left(1+\left\|a\right\|_{\dot{B}^{\frac{d}{p}}_{p,1}}\right)
\left(\left\|H^{\ell}\right\|_{\dot{B}^{\frac{d}{2}-1}_{2,1}}+\left\|H^{h}\right\|_{\dot{B}^{\frac{d}{p}-1}_{p,1}}\right),
\end{eqnarray*}
where we used the fact that $s_0\geq\frac{d}{p}-1$ due to $p\leq \frac{2d}{d-2}$. Consequently, we are led to the following inequality
\begin{eqnarray*}
&&\int_{0}^{t}\langle t-\tau\rangle^{-\frac{s_{0}+s}{2}}
\left\|\frac {1}{1+a}\left(\left(\nabla H^{h}\cdot H\right)^{T}-H\cdot\nabla H^{h}\right)\right\|^{\ell}_{\dot{B}^{-s_{0}}_{2,\infty}}d\tau\\
&\lesssim &\int_{0}^{t}\langle t-\tau\rangle^{-\frac{s_{0}+s}{2}}\left(\left\|H^{\ell}\right\|_{\dot{B}^{\frac{d}{2}-1}_{2,1}}
+\left\|H^{\ell}\right\|_{\dot{B}^{1-\frac{d}{p}}_{p,1}}+\left\|H^{h}\right\|_{\dot{B}^{1-\frac{d}{p}}_{p,1}}\right)\left\| H^{h}\right\|_{\dot{B}^{\frac{d}{p}}_{p,1}}d\tau.
\end{eqnarray*}
Since $p>d$, the interpolation inequality gives that
\begin{equation*}
\left\|\left(u,H\right)^{h}\right\|_{\dot{B}^{1-\frac{d}{p}}_{p,1}}\lesssim
\left\|\left(u,H\right)^{h}\right\|^{\frac{d}{p}}_{\dot{B}^{\frac{d}{p}-1}_{p,1}}
\left\|\left(u,H\right)^{h}\right\|^{1-\frac{d}{p}}_{\dot{B}^{\frac{d}{p}+1}_{p,1}},
\end{equation*}
which leads, thanks to the definition of $\mathcal{D}_{p}(t)$, to
\begin{equation}\label{Eq:3-29}
\left\|\left(u,H\right)^{h}(\tau)\right\|_{\dot{B}^{1-\frac{d}{p}}_{p,1}}\lesssim \tau^{-\left(1-\frac{d}{p}\right)}\langle\tau\rangle^{-\frac{d}{p}\alpha}\mathcal{D}_{p}(\tau)
\end{equation}
and
\begin{equation*}
\left\|H^{\ell}(\tau)\right\|_{\dot{B}^{\frac{d}{2}-1}_{2,1}}\lesssim \langle\tau\rangle^{-\left(\frac{d}{p}-\frac{1}{2}\right)}\mathcal{D}_{p}(\tau),
\end{equation*}
for all $\tau\geq0$. Consequently, we obtain
\begin{eqnarray*}
&&\int_{0}^{t}\langle t-\tau\rangle^{-\frac{s_{0}+s}{2}}\left\|\frac {1}{1+a}\left(\left(\nabla H^{h}\cdot H\right)^{T}-H\cdot\nabla H^{h}\right)\right\|^{\ell}_{\dot{B}^{-s_{0}}_{2,\infty}}d\tau\\
&\lesssim & \mathcal{D}^{2}_{p}(t)\int_{0}^{t}\langle t-\tau\rangle^{-\frac{s_{0}+s}{2}}\left(\langle\tau\rangle^{-\left(\frac{d}{p}-\frac{1}{2}\right)}
+\tau^{-\left(1-\frac{d}{p}\right)}\langle\tau\rangle^{-\frac{d}{p}\alpha}\right)
\tau^{-\frac{1}{2}}\langle\tau\rangle^{-\frac{\alpha}{2}}d\tau,
\end{eqnarray*}
where we have noticed that $\frac{d}{p}-\frac12<\frac12$. Recall $\alpha>1$ and $p>d$, it is not difficult to know that
$\frac{\alpha+3}{2}+\frac{d}{p}\left(\alpha-1\right)>\frac{\alpha}{2}+\frac{d}{p}>1$ and $0<\frac{3}{2}-\frac{d}{p}<1$. Inequality \eqref{Eq:3-3} thus ensures that
\begin{eqnarray*}
\int_{0}^{t}\langle t-\tau\rangle^{-\frac{s_{0}+s}{2}}
\left\|\frac {1}{1+a}\left(\left(\nabla H^{h}\cdot H\right)^{T}-H\cdot\nabla H^{h}\right)\right\|^{\ell}_{\dot{B}^{-s_{0}}_{2,\infty}}d\tau\lesssim \langle t\rangle^{-\frac{s_{0}+s}{2}}\mathcal{D}^{2}_{p}(t).
\end{eqnarray*}
Regarding the term with $-H\left(\mathrm{div}\,u^{h}\right)$, we note that, thanks to \eqref{Eq:3-27},
\begin{eqnarray*}
&&\int_{0}^{t}\langle t-\tau\rangle^{-\frac{s_{0}+s}{2}}\left\|H\left(\mathrm{div}\,u^{h}\right)\right\|_{\dot{B}^{-s_{0}}_{2,\infty}}d\tau\\
&\lesssim &\int_{0}^{t}\langle t-\tau\rangle^{-\frac{s_{0}+s}{2}}\left(\left\|H^{\ell}\right\|_{\dot{B}^{\frac{d}{p}}_{2,1}}+\left\|H^{\ell}\right\|_{\dot{B}^{1-\frac{d}{p}}_{p,1}}
+\left\|H^{h}\right\|_{\dot{B}^{1-\frac{d}{p}}_{p,1}}\right)\left\|\mathrm{div}\,u^{h}\right\|_{\dot{B}^{\frac{d}{p}-1}_{p,1}}d\tau.
\end{eqnarray*}
Therefore, combining with \eqref{Eq:3-3}, \eqref{Eq:3-20}, \eqref{Eq:3-23}, \eqref{Eq:3-28} and \eqref{Eq:3-29},
\begin{eqnarray*}
&&\int_{0}^{t}\langle t-\tau\rangle^{-\frac{s_{0}+s}{2}}\left\|H\left(\mathrm{div}\,u^{h}\right)\right\|_{\dot{B}^{-s_{0}}_{2,\infty}}d\tau\\
&\lesssim & \mathcal{D}^{2}_{p}(t)\int_{0}^{t}\langle t-\tau\rangle^{-\frac{s_{0}+s}{2}}\left(\langle\tau\rangle^{-\min\left(\frac{3d}{2p}-\frac{d}{4},\frac{1}{2}\right)}
+\tau^{-\left(1-\frac{d}{p}\right)}\langle\tau\rangle^{-\frac{d}{p}\alpha}\right)
\tau^{-\frac{1}{2}}\langle\tau\rangle^{-\frac{\alpha}{2}}d\tau\\
&\lesssim& \langle t\rangle^{-\frac{s_{0}+s}{2}} \mathcal{D}^{2}_{p}(t).
\end{eqnarray*}
The terms $H\cdot \nabla u^{h}$ and $-u\cdot \nabla H^{h}$ being completely similar to $-H\mathrm{div}\,u^{h}$.

Putting together all above estimates lead to the \eqref{Eq:3-16}. Then, combining with \eqref{Eq:3-15} for bounding the term of \eqref{Eq:3-5} pertaining to the data, we deduce that
\begin{equation} \label{Eq:3-30}
\langle t\rangle^{\frac{s_{0}+s}{2}}\left\|(a,u,H)(t)\right\|^{\ell}_{\dot{B}^{s}_{2,1}}
\lesssim \mathcal{D}_{p,0}+\mathcal{D}^{2}_{p}(t)+\mathcal{E}^{2}_{p}(t) \ \ \hbox{for all} \ \ t\geq0,
\end{equation}
provided that $-s_{0}<s\leq\frac{d}{2}+1$.

\subsection{Decay estimates for the high frequencies of  $\left(\nabla a,u,H\right)$}
In this section, we shall employ the $L^{p}$ energy approach
in terms of the effective velocity, to bound the second term of $\mathcal{D}_{p}(t)$.
The basic idea was initiated by Hoff and developed by Haspot for compressible Navier-Stokes equations.
For that purpose, we first recall the following elementary fact, which has been proved by \cite{DX}.
\begin{lem}\label{Lem3.3}
 Let $X:[0,T]\rightarrow \mathbb{R}_{+}$ be a continuous function. Suppose that $X^{p}$ is differentiable for some $p\geq1$ and fulfills
\begin{eqnarray*}
\frac{1}{p}\frac{d}{dt}X^{p}+QX^{p}\leq KX^{p-1}
\end{eqnarray*}
for some constant $Q\geq0$ and measurable function $K:[0,T]\rightarrow\mathbb{R}_{+}$.

Define $X_{\varepsilon}=\left(X^{p}+\varepsilon^{p}\right)^{\frac{1}{p}}$ for $\varepsilon>0$. Then it holds that
\begin{eqnarray*}
\frac{1}{p}\frac{d}{dt}X_{\varepsilon}+Q X_{\varepsilon}\leq K+Q\varepsilon.
\end{eqnarray*}
\end{lem}
Let $\mathcal{P}\triangleq \mathcal{I}_{d}+\nabla\left(-\Delta\right)^{-1}\mathrm{div}$
be the Leray projector onto divergence-free vector-fields. It follows from \eqref{Eq:1-6}
that $\mathcal{P}u$ fulfills
\begin{equation} \label{Eq:3-31}
\partial _{t}\mathcal{P}u-\mu_{\infty} \Delta \mathcal{P}u-I\cdot\nabla H=\mathcal{P}g.
\end{equation}
where we used the equality $\mathcal{P}(I\cdot\nabla H)=I\cdot\nabla H$. Applying the operator $\dot{\Delta}_{k}$ to \eqref{Eq:3-31} gives for all $k\in\mathbb{Z}$,
\begin{equation}\label{Eq:3-32}
\partial _{t}\mathcal{P}u_{k}-\mu_{\infty} \Delta \mathcal{P}u_{k}-I\cdot\nabla H_{k}=\mathcal{P}g_{k},
\end{equation}
where $u_{k}\triangleq\dot{\Delta}_{k}u, \ H_{k}\triangleq\dot{\Delta}_{k}H $\ and  $ g_{k}\triangleq\dot{\Delta}_{k}g.$

Multiplying each component of equation \eqref{Eq:3-32} by $|(\mathcal{P}u_{k})^{j}|^{p-2}(\mathcal{P}u)^{j}$ and integrating over $\mathbb{R}^{d}$ yields for $j=1,\, 2,\, \cdots,\, d$,
\begin{eqnarray*}
&&\frac{1}{p}\frac{d}{dt}\left\|\left(\mathcal{P}u_{k}\right)^{j}\right\|^{p}_{L^{p}}
-\mu_{\infty}\int\Delta\left(\mathcal{P}u_{k}\right)^{j}\left|\left(\mathcal{P}u_{k}\right)^{j}\right|^{p-2}\left(\mathcal{P}u_{k}\right)^{j}dx\\
&=&\int\mathcal{P}g^{j}_{k}\left|\left(\mathcal{P}u_{k}\right)^{j}\right|^{p-2}\left(\mathcal{P}u_{k}\right)^{j}dx+\int\left(I\cdot\nabla H_{k}\right)^{j}\left|\left(\mathcal{P}u_{k}\right)^{j}\right|^{p-2}\left(\mathcal{P}u_{k}\right)^{j}dx.
\end{eqnarray*}
The key observation is that the second term of the l.h.s., although not spectrally localized, may be handled from Proposition \ref{Prop2.5}.
After summation on $j=1,\, 2,\, \cdots,\, d$, we end up for some constant $c_{p}$ depending only $p$ with
\begin{equation*}
\frac{1}{p}\frac{d}{dt}\left\|\mathcal{P}u_{k}\right\|^{p}_{L^{p}}+c_{p}\mu_{\infty}2^{2k}\left\|\mathcal{P}u_{k}\right\|^{p}_{L^{p}}
\leq\left(\left\|\mathcal{P}g_{k}\right\|_{L^{p}}+\left\|\nabla H_{k}\right\|_{L^{p}}\right)\left\|\mathcal{P}u_{k}\right\|^{p-1}_{L^{p}}.
\end{equation*}
Hence, using the notation $\|\cdot\|_{\varepsilon,L^{p}}\triangleq(\|\cdot\|^{p}_{\varepsilon,L^{p}}+\varepsilon^{p})^{\frac{1}{p}}$, it
follows from Lemma \ref{Lem3.3} and Bernstein inequality that for all $\varepsilon>0$
\begin{equation}\label{Eq:3-33}
\frac{d}{dt}\left\|\mathcal{P}u_{k}\right\|_{\varepsilon,L^{p}}+c_{p}\mu_{\infty}2^{2k}\left\|\mathcal{P}u_{k}\right\|_{\varepsilon,L^{p}}
\leq\left\|\mathcal{P}g_{k}\right\|_{L^{p}}+C2^{k}\left\|H_{k}\right\|_{L^{p}}+c_{p}\mu_{\infty}2^{2k}\varepsilon.
\end{equation}

Here, following from Haspot's approach in \cite{HB}, we introduce the effective velocity $w$ defined by \eqref{Eq:3-1}. Then the
magnetic field $H$ satisfies
$$
\partial _{t} H-\Delta H =m-\left(\mathrm{div}\,w\right)I+I\cdot\nabla w-aI-I\cdot\nabla^2(-\Delta)^{-1}a+I\cdot\nabla\mathcal{P}u.
$$
Consequently, arguing exactly as for proving \eqref{Eq:3-33}, it follows that
\begin{eqnarray}
\frac{d}{dt}\left\|H_{k}\right\|_{\varepsilon,L^{p}}+c_{p}2^{2k}\left\|H_{k}\right\|_{\varepsilon,L^{p}}
&\leq&\left\|m_{k}\right\|_{L^{p}}+C2^{k}\left(\left\|w_{k}\right\|_{L^{p}}+\left\|\mathcal{P}u_{k}\right\|_{L^{p}}\right)\nonumber\\
\label{Eq:3-34}
&&+C2^{-k}\left\|\nabla a_{k}\right\|_{L^{p}}+c_{p}2^{2k}\varepsilon.
\end{eqnarray}
Additionally, we observe that $(a,w)$ satisfies
\begin{equation*}
\left\{
\begin{array}{l}
\partial _{t}w-\Delta w=\nabla\left(-\Delta \right)^{-1}\left(f-\mathrm{div}\,g\right) +w-\left(-\Delta \right)^{-1}\nabla a-\nabla\left(I\cdot H\right), \\
\partial _{t}a+a=f-\mathrm{div}\,w,
\end{array}
\right.
\end{equation*}
where we know $-\nabla(-\Delta)^{-1}\div (I\cdot\nabla H)=0$ as $\div H=0$. So arguing as for $\mathcal{P}u$, we can arrive at
\begin{eqnarray}
&&\frac{d}{dt}\left\|w_{k}\right\|_{\varepsilon,L^{p}}+c_{p}2^{2k}\left\|w_{k}\right\|_{\varepsilon,L^{p}}
\nonumber\\
&\leq& \left\|\nabla\left(-\Delta\right)^{-1}\left(f_{k}-\mathrm{div}\,g_{k}\right)\right\|_{L^{p}}
+\left\|w_{k}-\left(-\Delta\right)^{-1}\nabla a_{k}\right\|_{L^{p}}\nonumber\\
\label{Eq:3-35}
&&+C2^{k}\left\|H_{k}\right\|_{L^{p}}+c_{p}2^{2k}\varepsilon.
\end{eqnarray}
Moreover, it follows from \cite{DX} that
\begin{eqnarray*}
\frac{d}{dt}\left\|\nabla a_{k}\right\|_{\varepsilon,L^{p}}+\left\|\nabla a_{k}\right\|_{\varepsilon,L^{p}}
&\leq&\frac{1}{p}\left\| \mathrm{div}\,u\right\|_{L^\infty }\left\|a_{k}\right\|_{L^{p}}
+\left\|\nabla\dot{\Delta}_{k}(a\,\mathrm{div}\,u)\right\|_{L^{p}} \nonumber\\
&&+C2^{2k}\left\|w_{k}\right\|_{L^{p}}+\left\|R_{k}\right\|_{L^{p}}+\varepsilon,
\end{eqnarray*}
where we  denote $\dot{R}^{j}_{k}\triangleq \left[u\cdot \nabla ,\partial_{j}\dot{\Delta}_{k}\right]a$ for $j=1,\, 2,\, \cdots,\, d$.
Adding up that inequality (multiplied by $\eta c_{p}$) for some $\eta>0$ to \eqref{Eq:3-33}, \eqref{Eq:3-34} and \eqref{Eq:3-35} gives
\begin{eqnarray}
&&\frac{d}{dt}\left(\left\|\mathcal{P}u_{k}\right\|_{\varepsilon,L^{p}}+\left\|w_{k}\right\|_{\varepsilon,L^{p}}+\left\|H_{k}\right\|_{\varepsilon,L^{p}}+\eta c_{p}\left\|\nabla a_{k}\right\|_{\varepsilon,L^{p}}\right) \nonumber \\
&&+c_{p}2^{2k}\left(\mu_{\infty}\left\|\mathcal{P}u_{k}\right\|_{\varepsilon,L^{p}}
+\left\|w_{k}\right\|_{\varepsilon,L^{p}}+\left\|H_{k}\right\|_{\varepsilon,L^{p}}\right)
+\eta c_{p}\left\|\nabla a_{k}\right\|_{\varepsilon,L^{p}} \nonumber \\
&\leq&\left\|\mathcal{P}g_{k}\right\|_{L^{p}}+\left\|\nabla(-\Delta)^{-1}\left(f_{k}-\mathrm{div}\,g_{k}\right)\right\|_{L^{p}}
+\left\|m_{k}\right\|_{L^{p}}+\eta c_{p}\left(\frac{1}{p}\left\|\mathrm{div}\,u\right\|_{L^{\infty}}\left\|\nabla a_{k}\right\|_{L^{p}}\right. \nonumber \\
&&\left.+\left\|\nabla\dot{\Delta}_{k}\left(a\,\mathrm{div}\,u\right)\right\|_{L^{p}}+\left\|R_{k}\right\|_{L^{p}}\right)+N_\varepsilon+C\eta c_{p}2^{2k}\|w_{k}\|_{L^{p}}+C2^{-k}\left\|\nabla a_{k}\right\|_{L^{p}}\nonumber\\
\label{Eq:3-36}
&&+\left\|w_{k}-(-\Delta)^{-1}\nabla a_{k}\right\|_{L^{p}}
+C2^{k}\left(\left\|w_{k}\right\|_{L^{p}}+\left\|\mathcal{P}u_{k}\right\|_{L^{p}}
+2\left\|H_{k}\right\|_{L^{p}}\right),
\end{eqnarray}
where $N_\varepsilon\triangleq \left(c_{p}\mu_{\infty}2^{2k}+2c_{p}2^{2k}+\eta c_{p}\right)\varepsilon.$

As we know, $(-\Delta)^{-1}$ is a homogeneous Fourier multiplier of degree $-2$, and then we have
\begin{equation*}
\left\|\left(-\Delta \right)^{-1}\nabla a_{k}\right\| _{L^{p}}\lesssim 2^{-2k}\left\|\nabla a_{k}\right\| _{L^{p}}\lesssim 2^{-2k_{0}}\left\| \nabla a_{k}\right\| _{\varepsilon,L^{p}}
\ \ \hbox{for all} \ \  k\geq k_{0}-1.
\end{equation*}
Choosing $k_{0}$ suitably large and $\eta$ small enough, we deduce that the last six terms of the r.h.s. of \eqref{Eq:3-36} may be absorbed by the l.h.s. Hence, remembering that $f_{k}=\dot{\Delta}_{k}\mathrm{div}\left(au\right)$, we discover that there exist some $k_{0}\in\mathbb{Z}$ and $c_{0}>0$ such that for all $k\geq k_{0}-1$,
\begin{eqnarray*}
&&\frac{d}{dt}\left(\left\|\mathcal{P}u_{k}\right\|_{\varepsilon,L^{p}}+\left\|w_{k}\right\|_{\varepsilon,L^{p}}
+\left\|H_{k}\right\|_{\varepsilon,L^{p}}+\eta c_{p}\left\|\nabla a_{k}\right\|_{\varepsilon,L^{p}}\right)\\
&&+c_{0}\left(\left\|\mathcal{P}u_{k}\right\|_{\varepsilon,L^{p}}+\left\|w_{k}\right\|_{\varepsilon,L^{p}}+\left\|H_{k}\right\|_{\varepsilon,L^{p}}
+\eta c_{p}\left\|\nabla a_{k}\right\|_{\varepsilon,L^{p}}\right)\\
&\leq&\left\|\dot{\Delta}_{k}(au)\right\|_{L^{p}}+\left\|g_{k}\right\|_{L^{p}}+\left\|m_{k}\right\|_{L^{p}}+\eta c_{p}\left(\frac{1}{p}\left\|\mathrm{div}\,u\right\|_{L^{\infty}}\left\|\nabla a_{k}\right\|_{L^{p}}\right.\\
&&\left.+\left\|\nabla\dot{\Delta}_{k}\left(a\,\mathrm{div}\,u\right)\right\|_{L^{p}}+\left\|R_{k}\right\|_{L^{p}}\right)+N_\varepsilon.
\end{eqnarray*}
Then, integrating in time and having $\varepsilon$ tend to $0$, we are led to
\begin{eqnarray}
&&e^{c_{0}t}\left\|\left(\mathcal{P}u_{k},w_{k},H_{k},\nabla a_{k}\right)(t)\right\|_{L^{p}}\nonumber\\
\label{Eq:3-37}
&\lesssim &\left\|\left(\mathcal{P}u_{k},w_{k},H_{k},\nabla a_{k}\right)(0)\right\|_{L^{p}}+\int_{0}^{t}e^{c_{0}\tau}Z_{k}(\tau)d\tau
\end{eqnarray}
with $Z_{k}\triangleq Z^{1}_{k}+ \cdots + Z^{6}_{k}$ and
\begin{eqnarray*}
&&Z^{1}_{k}\triangleq\left\|\dot{\Delta}_{k}\left(au\right)\right\|_{L^{p}},\  \ \ \ \ \ \ \ \, \, \,
Z^{2}_{k}\triangleq\left\|g_{k}\right\|_{L^{p}},\  \ \ \ Z^{3}_{k}\triangleq\left\|m_{k}\right\|_{L^{p}},\\
&&Z^{4}_{k}\triangleq\left\|\nabla \dot{\Delta}_{k}\left(a\,\mathrm{div}\,u\right)\right\|_{L^{p}},\ \ Z^{5}_{k}\triangleq\left\|R_{k}\right\|_{L^{p}},\ \  \ \ Z^{6}_{k}\triangleq\left\|\mathrm{div}\,u\right\|_{L^{\infty}}\left\|\nabla a_{k}\right\|_{L^{p}}.
\end{eqnarray*}
It follows the definition of $w$ that
\begin{eqnarray*}
u=w-\nabla \left(-\Delta\right)^{-1} a+\mathcal{P}u,
\end{eqnarray*}
which leads for $k\geq k_{0}-1$ to
\begin{eqnarray*}
\left\|u_{k}-\left(w_{k}+\mathcal{P}u_{k}\right)\right\|_{L^{p}}\lesssim 2^{-2k_{0}}\left\|\nabla a_{k}\right\|_{L^{p}}.
\end{eqnarray*}
So we easily see that $\left(\nabla a_{k}, u_{k},H_{k}\right)$ fulfills a similar inequality as \eqref{Eq:3-37}.
Furthermore, it is shown that there exists a constant $c_{0}>0$ such that for all $k\geq k_{0}-1$ and $t\geq 0$,
\begin{equation}\label{Eq:3-38}
\left\|\left(\nabla a_{k},u_{k},H_{k}\right)(t)\right\|_{L^{p}}\lesssim e^{-c_{0}t}\left\|\left(\nabla a_{k},u_{k},H_{k}\right)(0)\right\|_{L^{p}}+\int_{0}^{t}e^{-c_{0}\left(t-\tau\right)}Z_{k}(\tau)d\tau.
\end{equation}
Multiplying both sides of \eqref{Eq:3-38} by $\langle t\rangle^{\alpha}2^{k\left(\frac{d}{p}-1\right)}$, taking the supremum on $[0,T]$ and then summing up over $k\geq k_{0}-1$ yields
\begin{eqnarray}
&&\left\|\langle t\rangle^{\alpha}(\nabla a,u,H)\right\|^{h}_{\tilde{L}^{\infty}_{T}(\dot{B}^{\frac{d}{p}-1}_{p,1})} \nonumber\\
\label{Eq:3-39}
\hspace{5mm}&\lesssim&\|(\nabla a_{0},u_{0},H_{0})\|^{h}_{\dot{B}^{\frac{d}{p}-1}_{p,1}}
+\sum_{k\geq k_{0}-1}\sup_{0\leq t\leq T}\left(\langle t\rangle^{\alpha}\int_{0}^{t}e^{-c_{0}(t-\tau)}2^{k\left(\frac{d}{p}-1\right)}Z_{k}(\tau)d\tau\right).
\end{eqnarray}
Firstly, we notice that
\begin{equation*}
\sum_{k\geq k_{0}-1}\sup_{0\leq t\leq 2}\left(\langle t\rangle^{\alpha}\int_{0}^{t}e^{-c_{0}(t-\tau)}2^{k\left(\frac{d}{p}-1\right)}Z_{k}(\tau)d\tau\right)
\lesssim \int_{0}^{2}\sum_{k\geq k_{0}-1}2^{k\left(\frac{d}{p}-1\right)}Z_{k}(\tau)d\tau.
\end{equation*}
Furthermore, it follows from Propositions \ref{Prop2.2} and \ref{Prop2.6} that
\begin{equation}\label{Eq:3-40}
\int_{0}^{2}\sum_{k\geq k_{0}-1}2^{k\left(\frac{d}{p}-1\right)}Z_{k}(\tau)d\tau
\lesssim \int_{0}^{2}\left(\left\|\left(au,g,m\right)\right\|^{h}_{\dot{B}^{\frac{d}{p}-1}_{p,1}}
+\left\|\nabla u\right\|_{\dot{B}^{\frac{d}{p}}_{p,1}}\left\|a\right\|_{\dot{B}^{\frac{d}{p}}_{p,1}}\right)d\tau.
\end{equation}
It is obvious that the last term of the r.h.s. of \eqref{Eq:3-40} may be bounded by $C\mathcal{E}^{2}_{p}(2)$ and that, thanks to Proposition \ref{Prop2.2}, we arrive at
\begin{equation*}
\left\|au\right\|_{L_{t}^{1}(\dot{B}_{p,1}^{\frac {d}{p}-1})}^{h}\lesssim \left\|au\right\|_{L_{t}^{1} (\dot{B}_{p,1}^{\frac {d}{p}})}
\lesssim \left\|a\right\|_{L_{t}^{2} (\dot{B}_{p,1}^{\frac {d}{p}})}\left\|u\right\|_{L_{t}^{2} (\dot{B}_{p,1}^{\frac {d}{p}})}.
\end{equation*}
Additionally, applying Propositions \ref{Prop2.2} and \ref{Prop2.4} yields
\begin{eqnarray*}
\left\|g\right\|_{L_{t}^{1}(\dot{B}_{p,1}^{\frac {d}{p}-1})}^{h}
&\lesssim&\left(\left\|u\right\|_{L_{t}^{\infty} (\dot{B}_{p,1}^{\frac {d}{p}-1})}\left\|\nabla u\right\|_{L_{t}^{1} (\dot{B}_{p,1}^{\frac {d}{p}})}+\left\|a\right\|_{L_{t}^{\infty} (\dot{B}_{p,1}^{\frac {d}{p}})}\left\|\nabla u\right\|_{L_{t}^{1} (\dot{B}_{p,1}^{\frac {d}{p}})}\right.\\
&&+\left\|a\right\|_{L_{t}^{2} (\dot{B}_{p,1}^{\frac {d}{p}})}\left\|\nabla a\right\|_{L_{t}^{2} (\dot{B}_{p,1}^{\frac {d}{p}-1})}
+\left\|a\right\|_{L_{t}^{2} (\dot{B}_{p,1}^{\frac {d}{p}})}\left\|\nabla H\right\|_{L_{t}^{2} (\dot{B}_{p,1}^{\frac {d}{p}-1})}\\
&&\left.+\left\|H\right\|_{L_{t}^{\infty} (\dot{B}_{p,1}^{\frac {d}{p}-1})}\left\|\nabla H\right\|_{L_{t}^{1} (\dot{B}_{p,1}^{\frac {d}{p}})}\right),\\
\left\|m\right\|^{h} _{L_{t}^{1}(\dot{B}_{p,1}^{\frac {d}{p}-1})} &\lesssim &\left(\left\| H\right\| _{L_{t}^{\infty} (\dot{B}_{p,1}^{\frac {d}{p}-1})}\left\|\nabla u\right\|_{L_{t}^{1}(\dot{B}_{p,1}^{\frac{d}{p}})}+\left\|u\right\|
_{L_{t}^{\infty}(\dot{B}_{p,1}^{\frac {d}{p}-1})}\left\|\nabla H\right\|_{L_{t}^{1}(\dot{B}_{p,1}^{\frac {d}{p}})}\right).
\end{eqnarray*}
By employing interpolation equality, it is easy to know that $\left\|(a,u,H)\right\|_{L_{t}^{2} (\dot{B}_{p,1}^{\frac {d}{p}})}\lesssim \mathcal{E}_{p}(t)$.
Furthermore, we can conclude that the first term in the right-hand side of \eqref{Eq:3-40} may be bounded by $\mathcal{E}^{2}_{p}(2)$. We end up with
\begin{equation}\label{Eq:3-41}
\sum_{k\geq k_{0}-1}\sup_{0\leq t\leq 2}\left(\langle t\rangle^{\alpha}\int_{0}^{t}e^{-c_{0}(t-\tau)}2^{k\left(\frac{d}{p}-1\right)}Z_{k}(\tau)d\tau\right)
\lesssim \mathcal{E}^{2}_{p}(2).
\end{equation}
Secondly, let us handle the supremum for $2\leq t\leq T$ in the last term of \eqref{Eq:3-39}, supposing (without loss of generality) that $T\geq2$. To do this, one can split the integral on $[0,t]$ into integrals $[0,1]$ and $[1,t]$. The $[0,1]$ part of the integral
can be bounded exactly as the supremum on $[0,2]$ treated before. Indeed, owing to $e^{-c_{0}(t-\tau)}\leq e^{-\frac{c_{0}t}{2}}$ for $2\leq t \leq T$ and $0\leq \tau \leq 1$, we can obtain
\begin{eqnarray*}
&&\sum_{k\geq k_{0}-1}\sup_{2\leq t\leq T}\left(\langle t\rangle^{\alpha}\int_{0}^{1}e^{-c_{0}(t-\tau)}2^{k\left(\frac{d}{p}-1\right)}Z_{k}(\tau)d\tau\right)\\
&\lesssim&\sum_{k\geq k_{0}-1}\sup_{2\leq t\leq T}\left(\langle t\rangle^{\alpha}e^{-\frac{c_{0}}{2}t}\int_{0}^{1}2^{k\left(\frac{d}{p}-1\right)}Z_{k}(\tau)d\tau\right)\lesssim \int_{0}^{1} \sum_{k\geq k_{0}-1} 2^{k\left(\frac{d}{p}-1\right)}Z_{k}(\tau)d\tau.
\end{eqnarray*}
Hence, following the procedure leading to \eqref{Eq:3-41}, we finally arrive at
\begin{equation} \label{Eq:3-42}
\sum_{k\geq k_{0}-1}\sup_{2\leq t\leq T}\left(\langle t\rangle^{\alpha}\int_{0}^{1}e^{-c_{0}(t-\tau)}2^{k\left(\frac{d}{p}-1\right)}Z_{k}(\tau)d\tau\right)
\lesssim \mathcal{E}^{2}_{p}(1).
\end{equation}
In order to handle the integral on $[1,t]$ for $2\leq t\leq T$, we use \eqref{Eq:3-3} to get
\begin{eqnarray}
&&\sum_{k\geq k_{0}-1}\sup_{2\leq t\leq T}\left(\langle t\rangle^{\alpha}\int_{1}^{t}e^{-c_{0}(t-\tau)}2^{k(\frac{d}{p}-1)}Z_{k}(\tau)d\tau\right)
\nonumber\\
\label{Eq:3-43}
&\lesssim& \sum_{k\geq k_{0}-1} 2^{k(\frac{d}{p}-1)}\sup_{1\leq t\leq T}t^{\alpha}Z_{k}(t).
\end{eqnarray}
In what follows, we claim the following two inequalities
\begin{equation} \label{Eq:3-44}
\left\|\tau \nabla \left(u,H\right)\right\|_{\tilde{L}^{\infty}_{t}(\dot{B}^{\frac{d}{p}}_{p,1})}\lesssim \mathcal{D}_{p}(t),\ \ \ \left\|\tau^{\alpha-1}\left(\nabla a,u,H\right)\right\|_{\tilde{L}^{\infty}_{t}(\dot{B}^{\frac{d}{p}-1}_{p,1})}
\lesssim\mathcal{D}_{p}(t).
\end{equation}
Indeed, it is clear that
\begin{equation*}
\left\|\tau\nabla\left(u,H\right)\right\|^{h}_{\tilde{L}^{\infty}_{t}(\dot{B}^{\frac{d}{p}}_{p,1})}\lesssim\mathcal{D}_{p}(t),\ \ \ \left\|\tau^{\alpha-1}\left(\nabla a,u,H\right)\right\|^{h}_{\tilde{L}^{\infty}_{t}(\dot{B}^{\frac{d}{p}-1}_{p,1})}
\lesssim\mathcal{D}_{p}(t).
\end{equation*}
On the other hand, for sufficiently small $\varepsilon$, we have
\begin{eqnarray*}
\left\|\tau \nabla (u,H)\right\|^{\ell}_{\tilde{L}^{\infty}_{t}(\dot{B}^{\frac{d}{p}}_{p,1})}
&\lesssim &\left\|\tau (u,H)\right\|^{\ell}_{\tilde{L}^{\infty}_{t}(\dot{B}^{\frac{d}{2}+1}_{2,1})}
\lesssim \left\|\tau (u,H)\right\|^{\ell}_{L^{\infty}_{t}(\dot{B}^{\frac{d}{2}+1-2\varepsilon}_{2,1})}\\
&\lesssim &\left\|\langle \tau\rangle^{\frac{d}{p}+\frac{1}{2}-\varepsilon}(u,H)\right\|^{\ell}_{L^{\infty}_{t}(\dot{B}^{\frac{d}{2}+1-2\varepsilon}_{2,1})}
\lesssim \mathcal{D}_{p}(t),\\
\left\|\tau^{\alpha-1}\left(\nabla a,u,H\right)\right\|^{\ell}_{\tilde{L}^{\infty}_{t}(\dot{B}^{\frac{d}{p}-1}_{p,1})}
&\lesssim &\left\|\tau^{\alpha-1}(a,u,H)\right\|^{\ell}_{\bar{L}^{\infty}_{t}(\dot{B}^{\frac{d}{2}-1}_{2,1})}\\
&\lesssim& \left\|\tau^{\alpha-1}(a,u,H)\right\|^{\ell}_{L^{\infty}_{t}(\dot{B}^{\frac{d}{2}-1-2\varepsilon}_{2,1})} \\
&\lesssim &\left\|\langle\tau\rangle^{\frac{s_{0}}{2}+\frac{d}{4}-\frac{1}{2}-\varepsilon}(a,u,H)\right\|^{\ell}_{L^{\infty}_{t}(\dot{B}^{\frac{d}{2}-1-2\varepsilon}_{2,1})}
\lesssim \mathcal{D}_{p}(t).
\end{eqnarray*}

To bound the right-side of \eqref{Eq:3-43}, it only need to estimate the new nonlinear terms (say, $g_{5}$, $g_{6}$ and $m$), which are not available in the usual compressible Navier-Stokes equations. It follows from Propositions \ref{Prop2.2}, \ref{Prop2.4} and \eqref{Eq:3-4}, \eqref{Eq:3-44} that
\begin{eqnarray*}
\left\|t^{\alpha}g_{5}\left(a,H\right)\right\|^{h}_{\tilde{L}^{\infty}_{T}(\dot{B}^{\frac{d}{p}-1}_{p,1})}
&\lesssim &\left\|t^{\alpha}\pi_{1}(a)\left(\nabla\left(I\cdot H\right)-I\cdot\nabla H\right)\right\|_{\tilde{L}^{\infty}_{T}(\dot{B}^{\frac{d}{p}}_{p,1})}\\
&\lesssim&\left\|t^{\alpha-1}a\right\|_{\tilde{L}^{\infty}_{T}(\dot{B}^{\frac{d}{p}}_{p,1})}
\left\|t \nabla H\right\|_{\tilde{L}^{\infty}_{T}(\dot{B}^{\frac{d}{p}}_{p,1})}\lesssim \mathcal{D}^{2}_{p}(T),\\
\left\|t^{\alpha}g_{6}\left(a,H\right)\right\|^{h}_{\tilde{L}^{\infty}_{T}(\dot{B}^{\frac{d}{p}-1}_{p,1})}
&\lesssim &\left\|\frac {t^{\alpha}}{1+a}\left(\left(\nabla H\cdot H\right)^{T}-H\cdot\nabla H\right)\right\|_{\tilde{L}^{\infty}_{T}(\dot{B}^{\frac{d}{p}-1}_{p,1})}\\
&\lesssim & (1+\left\|a\right\|_{\tilde{L}^{\infty}_{T}(\dot{B}^{\frac{d}{p}}_{p,1})})
\left\|t^{\alpha-1}H\right\|_{\tilde{L}^{\infty}_{T}(\dot{B}^{\frac{d}{p}-1}_{p,1})}
\left\|t \nabla H\right\|_{\tilde{L}^{\infty}_{T}(\dot{B}^{\frac{d}{p}}_{p,1})}\\
&\lesssim & \left\|t^{\alpha-1}H\right\|_{\tilde{L}^{\infty}_{T}(\dot{B}^{\frac{d}{p}-1}_{p,1})}
\left\|t\nabla H\right\|_{\tilde{L}^{\infty}_{T}(\dot{B}^{\frac{d}{p}}_{p,1})}\lesssim \mathcal{D}^{2}_{p}(T).
\end{eqnarray*}
Regarding those terms in $m\left(u,H\right)$:
$-H\left(\mathrm{div}\,u\right),  H\cdot\nabla u$ and  $-u\cdot\nabla H$,
with aid of Proposition \ref{Prop2.2} and \eqref{Eq:3-44}, we can get similarly
\begin{eqnarray*}
\left\|t^{\alpha}H\left(\mathrm{div}\,u\right)\right\|^{h}_{\tilde{L}^{\infty}_{T}(\dot{B}^{\frac{d}{p}-1}_{p,1})}
\lesssim \left\|t^{\alpha-1}H\right\|_{\tilde{L}^{\infty}_{T}(\dot{B}^{\frac{d}{p}-1}_{p,1})}
\left\|t\nabla u\right\|_{\tilde{L}^{\infty}_{T}(\dot{B}^{\frac{d}{p}}_{p,1})}\lesssim \mathcal{D}^{2}_{p}(T).
\end{eqnarray*}
Bounding $H\cdot\nabla u$ and $-u\cdot\nabla H$ follows from the same argument, and are thus omitted. Other terms, for example,
with $Z^{1}_{k},Z^{4}_{k},Z^{5}_{k}$ and $Z^{6}_{k}$ have been treated by the joint work with Danchin (see \cite{DX} for more details).
Consequently, putting all estimates together, we have
\begin{equation} \label{Eq:3-45}
\sum_{k\geq k_{0}-1}2^{k\left(\frac{d}{p}-1\right)}\sup_{1\leq t\leq T} t^{\alpha}Z_{k}(t)
\lesssim \mathcal{E}_{p}(T)\mathcal{D}_{p}(T)+\mathcal{D}^{2}_{p}(T).
\end{equation}
Plugging \eqref{Eq:3-45} in \eqref{Eq:3-43}, and remembering \eqref{Eq:3-39}, \eqref{Eq:3-41} and \eqref{Eq:3-42}, we conclude that
\begin{equation}\label{Eq:3-46}
\left\|\langle t\rangle^{\alpha}\left(\nabla a,u,H\right)\right\|^{h}_{\tilde{L}^{\infty}_{T}(\dot{B}^{\frac{d}{p}-1}_{p,1})}
\lesssim \left\|\left(\nabla a_{0},u_{0},H_{0}\right)\right\|^{h}_{\dot{B}^{\frac{d}{p}-1}_{p,1}}+\mathcal{E}^{2}_{p}(T)+\mathcal{D}^{2}_{p}(T).
\end{equation}

\subsection{Decay and gain of regularity for the high frequencies of $(u,H)$}
In order to estimate the last term in $\mathcal{D}_{p}(t)$, we shall take full advantage of the parabolic smoothing effect. To this end,
we see that the couple $(u,H)$ fulfills
\begin{equation}\label{Eq:3-47}
\left\{
\begin{array}{l}
\partial _{t}u-\mathcal{A}u=g-\nabla(I\cdot H)+I\cdot\nabla H-\nabla a, \\
\partial _{t}H-\Delta H=m-\left(\div u\right)I+I\cdot\nabla u.
\end{array}
\right.
\end{equation}
To obtain the desired estimates, we reformulate \eqref{Eq:3-47} as follows:
\begin{equation*}
\left\{
\begin{array}{l}
\partial _{t}\left(t\mathcal{A}u\right)-\mathcal{A}\left(t\mathcal{A}u\right)=t\mathcal{A}g-t\mathcal{A}\nabla(I\cdot H)+t\mathcal{A}(I\cdot\nabla H)+\mathcal{A}u-t\mathcal{A}\nabla a, \\
\partial _{t}\left(t\Delta H\right)-\Delta\left(t\Delta H\right)=t\Delta m-t\Delta\left(\mathrm{div}\,u\right)I+t\Delta\left(I\cdot\nabla u\right)+\Delta H\\
\left(t\mathcal{A}u,t\Delta H\right)|_{t=0}=\left(0,0\right).
\end{array}
\right.
\end{equation*}
Taking advantage of Proposition \ref{Prop2.7}, Remark \ref{Rem2.2} and Bernstein inequality, we have for $k\geq k_{0}-1$,
\begin{eqnarray}
\left\|\tau \nabla^{2}\left(u,H\right)\right\|^{h}_{\tilde{L}^{\infty}_{t}(\dot{B}^{\frac{d}{p}-1}_{p,1})}
&\lesssim &\left\|\left(u,H\right)\right\|^{h}_{L^{1}_{t}(\dot{B}^{\frac{d}{p}+1}_{p,1})}
+\left\|\tau\left(g,m\right)\right\|^{h}_{\tilde{L}^{\infty}_{t}(\dot{B}^{\frac{d}{p}-1}_{p,1})} \nonumber \\
\label{Eq:3-48}
&&+\left\|\tau \nabla a\right\|^{h}_{\tilde{L}^{\infty}_{t}(\dot{B}^{\frac{d}{p}-1}_{p,1})}
+C2^{-k_{0}}\left\|\tau\nabla\left(u,H\right)\right\|^{h}_{\tilde{L}^{\infty}_{t}(\dot{B}^{\frac{d}{p}}_{p,1})}.
\end{eqnarray}
It is obvious that the last term on the right side of \eqref{Eq:3-48} may be absorbed by the l.h.s. if $k_{0}$ is large enough.
Hence, we are led to
\begin{eqnarray}
\left\|\tau \nabla(u,H)\right\|^{h}_{\tilde{L}^{\infty}_{t}(\dot{B}^{\frac{d}{p}}_{p,1})}
&\lesssim&\left\|(u,H)\right\|^{h}_{L^{1}_{t}(\dot{B}^{\frac{d}{p}+1}_{p,1})}
+\left\|\tau\left(g,m\right)\right\|^{h}_{\tilde{L}^{\infty}_{t}(\dot{B}^{\frac{d}{p}-1}_{p,1})}\nonumber\\
\label{Eq:3-49}
&&+\left\|\tau \nabla a\right\|^{h}_{\tilde{L}^{\infty}_{t}(\dot{B}^{\frac{d}{p}-1}_{p,1})}.
\end{eqnarray}
According the bounds in Theorem \ref{Thm1.1}, we have
$$\left\|(u,H)\right\|^{h}_{L^{1}_{t}(\dot{B}^{\frac{d}{p}+1}_{p,1})}\lesssim \mathcal{E}_{p,0}. $$
Note that $\alpha>1$, it following from \eqref{Eq:3-46} that
$$\left\|\tau\nabla a\right\|^{h}_{\tilde{L}^{\infty}_{t}(\dot{B}^{\frac{d}{p}-1}_{p,1})}\leq
\left\|\langle \tau\rangle^{\alpha}  \nabla a\right\|^{h}_{\tilde{L}^{\infty}_{t}(\dot{B}^{\frac{d}{p}-1}_{p,1})}
\lesssim \left\|\left(\nabla a_{0},u_{0},H_{0}\right)\right\|^{h}_{\dot{B}^{\frac{d}{p}-1}_{p,1}}+\mathcal{E}^{2}_{p}(t)+\mathcal{D}^{2}_{p}(t).$$
Next, we bound the nonlinear terms with $g$ and $m$. In fact, the first four terms of $g$ have been dealt with in \cite{DX}:
\begin{eqnarray*}
&&\left\|\tau\, \pi_{2}(a)\nabla a\right\|^{h}_{\tilde{L}^{\infty}_{t}(\dot{B}^{\frac{d}{p}-1}_{p,1})}\lesssim \mathcal{D}^{2}_{p}(t),\\
&&\left\|\tau \,u\cdot \nabla u \right\|^{h}_{\tilde{L}^{\infty}_{t}(\dot{B}^{\frac{d}{p}-1}_{p,1})}
+\left\|\tau\, \pi_{1}(a)\mathcal{A}u\right\|^{h}_{\tilde{L}^{\infty}_{t}(\dot{B}^{\frac{d}{p}-1}_{p,1})}
\lesssim \mathcal{E}_{p}(t)\mathcal{D}_{p}(t),\\
&&\left\|\frac {\tau}{1+a}\mathrm{div}\left(2\widetilde{\mu }(a)\,D(u)+\widetilde{\lambda }(a)\,\mathrm{div}\,u\, \mathrm{Id}\right)\right\|^{h}_{\tilde{L}^{\infty}_{t}(\dot{B}^{\frac{d}{p}-1}_{p,1})}
\lesssim \mathcal{E}_{p}(t)\mathcal{D}_{p}(t).
\end{eqnarray*}
Product and compositions estimates (see Propositions \ref{Prop2.2} and \ref{Prop2.4}) adapted to title spaces give
\begin{eqnarray*}
\hspace{-6mm}\left\|\tau\,\pi_{1}(a)\left(\nabla\left(I\cdot H\right)-I\cdot\nabla H\right)\right\|^{h}_{\tilde{L}^{\infty}_{t}(\dot{B}^{\frac{d}{p}-1}_{p,1})}
&\lesssim& \left\|a\right\|_{\tilde{L}^{\infty}_{t}(\dot{B}^{\frac{d}{p}}_{p,1})}
\left\|\tau \nabla H\right\|_{\tilde{L}^{\infty}_{t}(\dot{B}^{\frac{d}{p}}_{p,1})}\\
&\lesssim&  \mathcal{E}_{p}(t) \mathcal{D}_{p}(t),\\
\left\|\frac {\tau}{1+a}\left(\left(\nabla H\cdot H\right)^{T}-H\cdot\nabla H\right)\right\|^{h}_{\tilde{L}^{\infty}_{t}(\dot{B}^{\frac{d}{p}-1}_{p,1})}
&\lesssim& \left\|H\right\|_{\tilde{L}^{\infty}_{t}(\dot{B}^{\frac{d}{p}-1}_{p,1})}
\left\|\tau \nabla H\right\|_{\tilde{L}^{\infty}_{t}(\dot{B}^{\frac{d}{p}}_{p,1})}\\
&\lesssim&  \mathcal{E}_{p}(t) \mathcal{D}_{p}(t).
\end{eqnarray*}
Let us finally consider those terms in $m$. We conclude that thanks to Proposition \ref{Prop2.2}, the definition of $\mathcal{E}_{p}(t)$ and \eqref{Eq:3-44},
\begin{eqnarray*}
&&\hspace{-5mm}\left\|\tau\left(H\cdot\nabla u-H\left(\mathrm{div}\,u\right)\right)\right\|^{h}_{\tilde{L}^{\infty}_{t}(\dot{B}^{\frac{d}{p}-1}_{p,1})}
\lesssim \left\|H\right\|_{\tilde{L}^{\infty}_{t}(\dot{B}^{\frac{d}{p}-1}_{p,1})}
\left\|\tau \nabla u\right\|_{\tilde{L}^{\infty}_{t}(\dot{B}^{\frac{d}{p}}_{p,1})}
\lesssim \mathcal{E}_{p}(t) \mathcal{D}_{p}(t),\\
&&\hspace{-5mm}\left\|\tau \,u\cdot \nabla H\right\|^{h}_{\tilde{L}^{\infty}_{t}(\dot{B}^{\frac{d}{p}-1}_{p,1})}
\lesssim \left\|u\right\|_{\tilde{L}^{\infty}_{t}(\dot{B}^{\frac{d}{p}-1}_{p,1})}
\left\|\tau \nabla H\right\|_{\tilde{L}^{\infty}_{t}(\dot{B}^{\frac{d}{p}}_{p,1})}
\lesssim \mathcal{E}_{p}(t) \mathcal{D}_{p}(t).
\end{eqnarray*}
Hence, reverting to \eqref{Eq:3-49}, we end up with
\begin{equation} \label{Eq:3-50}
\left\|\tau \nabla\left(u,H\right)\right\|^{h}_{\tilde{L}^{\infty}_{t}(\dot{B}^{\frac{d}{p}}_{p,1})}
\lesssim \mathcal{E}_{p,0}+\mathcal{E}^{2}_{p}(t)+\mathcal{D}^{2}_{p}(t).
\end{equation}
 Therefore, adding up \eqref{Eq:3-50} to \eqref{Eq:3-30} and  \eqref{Eq:3-46} yields for all $T\geq0$,
\begin{equation*}
\mathcal{D}_{p}(T)\lesssim \mathcal{D}_{p,0}+\left\|\left(a_{0},u_{0},H_{0}\right)\right\|^{\ell}_{\dot{B}^{\frac{d}{2}-1}_{2,1}}
+\left\|\left(\nabla a_{0},u_{0},H_{0}\right)\right\|^{h}_{\dot{B}^{\frac{d}{p}-1}_{p,1}}+\mathcal{E}^{2}_{p}(T)+\mathcal{D}^{2}_{p}(T).
\end{equation*}
As Theorem \ref{Thm1.1} ensures that $\mathcal{E}_{p}\lesssim\mathcal{E}_{p,0}\ll1$ and as
$$\left\|\left(a_{0},u_{0},H_{0}\right)\right\|^{\ell}_{\dot{B}^{\frac{d}{2}-1}_{2,1}}\lesssim \left\|\left(a_{0},u_{0},H_{0}\right)\right\|^{\ell}_{\dot{B}^{-s_{0}}_{2,\infty}},$$
one can conclude that \eqref{Eq:1-8} is satisfied for all time if $\mathcal{D}_{p,0}$ and $\left\|(a_{0},u_{0},H_{0})\right\|^{h}_{\dot{B}^{\frac{d}{p}-1}_{p,1}}$ are small enough. This completes the proof of Theorem \ref{Thm1.2}.
\subsection{The proof of Corollary 1.1}
\begin{proof}
In comparison with those efforts in \cite{DX}, it suffices to bound $H$ as an example. Recall that for functions with compactly supported Fourier transform, one has the embedding $\dot{B}^{s}_{2,1}\hookrightarrow\dot{B}_{p,1}^{s-d\left(\frac{1}{2}-\frac{1}{p}\right)}\hookrightarrow\dot{B}^{s}_{p,1}$
for the low frequencies, whenever $p\geq2$. Hence, we get for all $-s_{0}<s\leq\frac{d}{p}-1$,
\begin{equation*}
\sup_{t\in[0,T]}\langle t\rangle^{\frac{s_{0}+s}{2}}\left\|\Lambda^{s}H\right\|_{\dot{B}^{0}_{p,1}}
\lesssim\left\|\langle t\rangle^{\frac{s_{0}+s}{2}}H\right\|^{\ell}_{L^{\infty}_{T}(\dot{B}^{s}_{2,1})}+\left\|\langle t\rangle^{\frac{s_{0}+s}{2}}H\right\|^{h}_{L^{\infty}_{T}(\dot{B}^{s}_{p,1})}.
\end{equation*}
Taking advantage of \eqref{Eq:1-8} and the definition of functional $\mathcal{D}_{p}$, we see that
\begin{equation*}
\left\|\langle t\rangle^{\frac{s_{0}+s}{2}}H\right\|^{\ell}_{L^{\infty}_{T}(\dot{B}^{s}_{2,1})}
\lesssim\mathcal{D}_{p,0}+\left\|\left(\nabla a_{0},u_{0},H_{0}\right)\right\|^{h}_{\dot{B}^{\frac{d}{p}-1}_{p,1}}  \text{ if } -s_{0}<s\leq\frac{d}{2}+1
\end{equation*}
and that, owing to the fact that $\alpha\geq\frac{s_{0}+s}{2}$ for all $s\leq\frac{d}{p}-1$ (if $\varepsilon$ is sufficiently small),
\begin{equation*}
\left\|\langle t\rangle^{\frac{s_{0}+s}{2}}H\right\|^{h}_{L^{\infty}_{T}(\dot{B}^{s}_{p,1})}
\lesssim\mathcal{D}_{p,0}+\left\|\left(\nabla a_{0},u_{0},H_{0}\right)\right\|^{h}_{\dot{B}^{\frac{d}{p}-1}_{p,1}}.
\end{equation*}
Hence, the proof of Corollary \ref{Cor1.1} is complete.
\end{proof}

Corollary \ref{Cor1.2} follows from Corollary \ref{Cor1.1} and Gagliardo-Nirenberg type inequalities directly, see \cite{DX} for the same details.

\end{document}